\pdfoutput=1
\documentclass[11pt,reqno, a4paper]{amsart}
\usepackage{textcomp, mathrsfs}
\usepackage[expansion=false]{microtype}
\usepackage[english]{babel}
\usepackage{amsthm,amsmath}
\usepackage{amssymb}
\numberwithin{equation}{section}
\usepackage[dvipsnames]{xcolor}

\usepackage[T1]{fontenc}
\usepackage[utf8]{inputenc}
\usepackage{mathtools,apptools}
\usepackage[top=.8in, bottom=.8in, left=0.8in, right=0.8in]{geometry}
\usepackage{fancyhdr}
\usepackage{comment}
\usepackage[backref=page]{hyperref}
\hypersetup{breaklinks=true,hypertexnames=false,colorlinks=true,urlcolor=ForestGreen,citecolor=ForestGreen,linkcolor=blue!70!green,bookmarksnumbered,bookmarksopen}

\usepackage{graphicx}
\usepackage{subcaption}
\usepackage{booktabs}

\newtheorem{theorem}{Theorem}[section]
\newtheorem{corollary}[theorem]{Corollary}
\newtheorem{proposition}[theorem]{Proposition}
\newtheorem{definition}[theorem]{Definition}
\newtheorem{lemma}[theorem]{Lemma}
\newtheorem{remark}[theorem]{Remark}

\title[Nonlocal Tikhonov Regularization]{Nonlocal Tikhonov Regularization: Hilbert Scales, Explicit Rates, and the Classical Limit}
\author{Debangana Mukherjee}
\address{Department of Mathematics, School of Interwoven Arts and Science (SIAS), Krea University, Sricity, India}
\email{debangana.mukherjee@krea.edu.in}
\author{Akash Ashirbad Panda}
\address{Department of Mathematics, IIT Bhubaneswar, Bhubaneswar, India}
\email{akashpanda@iitbbs.ac.in}

\begin{document}

\begin{abstract}
We study fractional-Sobolev Tikhonov regularization for linear inverse
problems on a bounded Lipschitz domain. The regularization penalty is
generated by the restricted Dirichlet fractional Laplacian, and the
associated variational problem is shown to admit a unique minimizer that
depends Lipschitz continuously on the data. Identifying the positive
self-adjoint operator
$$A_s=I+(-\Delta)^s,\, D(A_s^{1/2})=H_0^s(\Omega),$$
we transform the problem isometrically into a classical Hilbert-space
Tikhonov problem with observation operator
$B=KA_s^{-1/2}$. This yields explicit mean-square error bounds and an
order-optimal \emph{a priori} and \emph{a posteriori} parameter rules
under H\"older-type source
conditions.

The framework is illustrated by partial observations and by the backward
fractional heat equation. In the latter case,
$$ B^*B=A_s^{-1}e^{-2tA_s}, $$
which permits a mode-wise description of the source condition, the
singular-value decay, and the effective reconstruction bandwidth. We also
study the local limit $s\to1^-$: after Bourgain--Brezis--Mironescu
normalization, the fractional functionals $\Gamma$-converge in
$L^2(\Omega)$ to the classical $H_0^1$-Tikhonov functional, and the
corresponding minimizers converge strongly in $L^2(\Omega)$. Numerical
experiments for the backward fractional heat problem illustrate the
reconstruction procedure and the influence of the penalty order, and
confirm the predicted mean-square convergence rate to within a few
percent via Monte Carlo simulation, with Morozov's discrepancy
principle attaining the same order-optimal rate a posteriori.
\end{abstract}

\maketitle


\smallskip



\section{Introduction}

\subsection*{Background and motivation}

Inverse problems seek to recover an unknown state from indirect and noisy
measurements, typically modeled by
\[
Ku=y^\delta,
\]
where $K$ is a bounded linear forward operator. When $K$ is compact or
smoothing, inversion is unstable and regularization is required; see
\cite{TikhonovArsenin77,EHN96,Vogel02,Kirsch11,KaipioSomersalo05,Scherzer09}.
Among the available approaches, Tikhonov regularization remains the most
widely used, combining a variational formulation with a quantitative
convergence theory that provides explicit reconstruction error estimates.
Classically, regularization is imposed through the
$H_0^1(\Omega)$-norm, thereby identifying the Laplacian as the operator
governing smoothness. While appropriate for locally diffusive phenomena,
this choice is less natural in applications such as anomalous transport,
image processing, peridynamics, finance, and nonlocal continuum mechanics,
where the underlying dynamics are intrinsically nonlocal. In these settings,
the natural energy spaces are fractional Sobolev spaces, whose Dirichlet
forms involve nonlocal interactions over $\Omega\times\Omega$ rather than
pointwise gradients; see
\cite{DiNezza12,CaffarelliSilvestre07,Kwa06,Lischke20}. This motivates
replacing the classical $H_0^1(\Omega)$-penalty by the fractional energy on
$H_0^s(\Omega)$, $0<s<1$, thereby aligning the regularization with the
nonlocal structure of the model.

Two terminological remarks are in order. First, the phrase
\emph{fractional Tikhonov regularization} already appears in numerical linear
algebra, where it refers to fractional-power modifications of the filter
function in discrete Tikhonov regularization
\cite{HochstenbachReichel2011,GerthKlannRamlauReichel2015,HochstenbachNoscheseReichel2015}.
Here, ``fractional'' instead refers to the order of the Sobolev space
$H_0^s(\Omega)$ defining the regularization penalty. Second, the
\emph{integral} (restricted) fractional Laplacian considered throughout this
paper is, in general, different from the \emph{spectral} fractional power of
the classical Dirichlet Laplacian; see
\cite{Kwa06,Lischke20}. Since these operators possess different analytical
and numerical properties, this distinction is essential. In particular, the
numerical scheme developed in
Section~\ref{sec:numerical-experiments} discretizes the integral fractional
Laplacian directly.

\subsection*{A motivating inverse problem}

A principal application throughout the paper is the backward fractional heat
equation. Let
\[
A_s=I+(-\Delta)^s
\]
denote the restricted Dirichlet fractional operator and suppose that
\[
y^\delta=e^{-tA_s}u^\dagger+\eta^\delta,
\qquad t>0.
\]
Since the semigroup $e^{-tA_s}$ exponentially attenuates high-frequency
components, recovering $u^\dagger$ from $y^\delta$ is a severely ill-posed
inverse problem. A distinctive feature of this model is that the same
operator $A_s$ governs both the forward evolution and the regularization
penalty. This structural compatibility allows the abstract Hilbert-scale
theory developed in Section~\ref{sec:statistical-error} to be made
completely explicit in Section~\ref{sec:application}, distinguishing this
example from the more general partial-observation problem considered
alongside it.

\subsection*{Analytical framework and relation to the literature}

The abstract theory of Hilbert-scale Tikhonov regularization was developed
by Natterer \cite{Natterer1984}, Neubauer
\cite{Neubauer1988}, and Nair, Pereverzev, and Tautenhahn
\cite{NairPereverzevTautenhahn2005}; see also
\cite{Hegland95,EnglNeubauer1999,Tautenhahn96}. This framework assumes a
positive self-adjoint operator generating the Hilbert scale and derives
source conditions and convergence rates relative to it. For a specific
nonlocal penalty, however, one must still identify the underlying operator
and characterize its domain. In our setting, this amounts to identifying
\[
A_s=I+(-\Delta)^s,
\qquad
D(A_s^{1/2})=H_0^s(\Omega),
\]
which provides the analytical foundation for the Hilbert-scale
reformulation developed in
Section~\ref{sec:statistical-error}.

Fractional Sobolev penalties have previously been considered in inverse
problems. A\ss{}mann and R\"osch \cite{AssmannRosch2015} established convergence
results under Sobolev source conditions, while Antil, Di, and Khatri
\cite{AntilDiKhatri2020} incorporated fractional Laplacian regularization
into a bilevel deep-learning framework for tomography. Related nonlocal
regularization techniques also appear in image processing
\cite{GilboaOsher2008}. Relative to these works, our contribution is an
explicit Hilbert-scale realization for the restricted fractional Laplacian,
together with order-optimal \emph{a priori} and Morozov
\emph{a posteriori} parameter-choice rules under a H\"older-type source
condition in the spirit of Hofmann and Yamamoto
\cite{HofmannYamamoto2010}. The backward fractional heat equation serves as
a representative example in which the forward operator and the
regularization penalty share the same spectral structure, allowing the
resulting convergence rates and effective reconstruction bandwidth to be
described explicitly.

The backward heat problem belongs to the class of severely ill-posed inverse
problems, where the singular values of the forward operator decay
exponentially; see \cite{Mair1994,Hohage2000}. In such problems,
unmatched regularization operators generally yield only logarithmic
convergence rates. By contrast, the matched structure of the present
framework recovers polynomial convergence rates despite the exponential
ill-posedness, and the mechanism behind this improvement becomes explicit at
the level of individual Fourier--$A_s$ modes.

Since the fractional order $s$ is itself a modeling parameter, we also study
the local limit $s\to1^-$ using the Bourgain--Brezis--Mironescu
characterization of $H^1$ as the limit of fractional Sobolev seminorms
\cite{BourgainBrezisMironescu2001,Ponce2004,Davila02}. After the appropriate
normalization, the fractional Tikhonov functionals
$\Gamma$-converge in $L^2(\Omega)$ to the classical
$H_0^1$-Tikhonov functional, and the corresponding minimizers converge
strongly. Thus, the proposed framework provides a continuous bridge between
nonlocal and classical regularization.

Rather than developing a new abstract Hilbert-scale theory, the present work
provides a concrete realization of the existing framework for the
restricted fractional Laplacian. The isometric transformation
\[
z=A_s^{1/2}u,
\qquad
B=KA_s^{-1/2},
\]
reduces the fractional Tikhonov problem to a standard Hilbert-space
formulation, making the classical spectral machinery directly applicable to
a genuinely nonlocal regularization problem.

\subsection*{Contributions}

The main contributions of this work are as follows.
\begin{itemize}
\item[(i)] We establish existence, uniqueness, weak Euler--Lagrange
characterization, and Lipschitz stability of the fractional Tikhonov
minimizer.
\item[(ii)] We identify the positive self-adjoint operator generated by the
fractional Dirichlet form and derive the Hilbert-scale reformulation through
$A_s^{1/2}$.
\item[(iii)] Under the H\"older-type source condition
$A_s^{1/2}u^\dagger=(B^*B)^\beta w$, we obtain explicit mean-square error
bounds and an order-optimal \emph{a priori} parameter choice; we further
show that Morozov's discrepancy principle attains the \emph{same}
order-optimal rate a posteriori, without requiring the smoothness
index $\beta$ to be known in advance.
\item[(iv)] We treat partial observations and the backward fractional heat
equation. For the latter, the identity
\begin{gather*}
B^*B=A_s^{-1}e^{-2tA_s}
\end{gather*}
yields a mode-wise source condition, singular-value asymptotics, and an
explicit effective bandwidth, showing how a matched regularization operator
upgrades the generic logarithmic rates of severely ill-posed problems to a
polynomial one.
\item[(v)] Using the Bourgain--Brezis--Mironescu limit, we prove pointwise
convergence of the energies, $\Gamma$-convergence of the extended
functionals in strong $L^2(\Omega)$, equicoercivity, and convergence of
minimizers as $s\to1^-$.
\item[(vi)] Numerical experiments for the backward fractional heat problem,
based on a direct discretization of the integral fractional Laplacian,
illustrate the reconstruction method, Morozov parameter selection ---
justified theoretically in item (iii) above --- and the role
of the fractional penalty order, and confirm the predicted mean-square
convergence rate to within a few percent via Monte Carlo simulation.
\end{itemize}

\subsection*{Organization of the paper}
Section~\ref{sec:var-form} develops the variational theory.
Section~\ref{sec:statistical-error} gives the Hilbert-scale transformation
and error estimates. Section~\ref{sec:application} treats the two
applications and the spectral structure of backward heat inversion.
Section~\ref{sec:Tikhonov-model} studies the classical limit and
\(\Gamma\)-convergence. Section~\ref{sec:numerical-experiments} presents the
numerical experiments, and Section~\ref{sec:conclusion} concludes the paper.

\section{Variational Analysis of the Fractional Tikhonov Functional}\label{sec:var-form}

Let $\Omega\subset\mathbb R^N$ be a bounded Lipschitz domain and
$s\in(0,1)$. We consider
\begin{equation*}
    y^\delta=Ku^\dagger+\eta^\delta,\, K:H_0^s(\Omega)\to L^2(\Omega),
\end{equation*}
where $K$ is bounded and linear and
$$\mathbb E\|\eta^\delta\|_{L^2(\Omega)}^2\le\delta^2.$$
Recall that \(J_\alpha\) denotes the fixed-\(s\) fractional Tikhonov
functional introduced in the notation subsection \
This section establishes the basic variational properties of the associated
fractional Tikhonov functional using the direct method and standard convex
analysis; see \cite{EkelandTemam99}.

\subsection{Problem formulation}
Since the inverse problem is generally ill posed, stable reconstruction
requires regularization. We therefore consider the fractional Tikhonov
functional, which is the natural analogue of the classical Tikhonov
functional obtained by replacing the standard
$H_0^1(\Omega)$-regularization with the fractional Sobolev norm,
\begin{equation*}
J_\alpha:H_0^s(\Omega)\longrightarrow\mathbb{R},
\,J_\alpha(u)
=\frac12\|Ku-y^\delta\|_{L^2(\Omega)}^2+\frac{\alpha}{2}\|u\|_{H_0^s(\Omega)}^2,
\end{equation*}
where $\alpha>0$ denotes the regularization parameter. The corresponding
variational problem is
\begin{equation}\label{eq:min_problem}
\min_{u\in H_0^s(\Omega)}
J_\alpha(u).
\end{equation}
The norm on $H_0^s(\Omega)$ is given by
\begin{equation*}
\|u\|_{H_0^s(\Omega)}^2=\|u\|_{L^2(\Omega)}^2+[u]_{H^s(\mathbb{R}^N)}^2,
\end{equation*}
where
$$[u]_{H^s(\mathbb{R}^N)}^2=\frac{C_{N,s}}{2}\iint_{\mathbb{R}^{2N}}\frac{|u(x)-u(y)|^2}{|x-y|^{N+2s}}
\,dx\,dy.$$
Throughout the paper, $(-\Delta)^s$ denotes the integral fractional
Laplacian with homogeneous exterior Dirichlet condition, associated
with the above quadratic form.
Using the definition of the fractional Sobolev norm,
$J_\alpha$ admits the equivalent representation
\begin{gather*}
J_\alpha(u)=\frac12\|Ku-y^\delta\|_{L^2(\Omega)}^2+\frac{\alpha}{2}
\|u\|_{L^2(\Omega)}^2+\frac{\alpha C_{N,s}}{4}\iint_{\mathbb{R}^{2N}}\frac{|u(x)-u(y)|^2}
{|x-y|^{N+2s}}
\,dx\,dy.
\end{gather*}
The following proposition establishes the basic analytical properties of
the fractional Tikhonov functional that are required for the subsequent
well-posedness analysis.
\begin{proposition}\label{prop:functional}
The functional $J_\alpha:H_0^s(\Omega)\to\mathbb R$ satisfies
\begin{enumerate}
\item $J_\alpha$ is continuous;
\item $J_\alpha$ is strictly convex;
\item $J_\alpha$ is coercive;
\item $J_\alpha$ is weakly lower semicontinuous.
\end{enumerate}
\end{proposition}
\begin{proof}
(1) Since $K:H_0^s(\Omega)\to L^2(\Omega)$ is bounded and linear, the
mapping
$u\mapsto \|Ku-y^\delta\|_{L^2(\Omega)}^2$
is continuous. Since $u\mapsto \|u\|_{H_0^s(\Omega)}^2$ is also
continuous, $J_\alpha$ is continuous.

(2) The mapping
$u\mapsto \|Ku-y^\delta\|_{L^2(\Omega)}^2$
is convex, while
$u\mapsto \|u\|_{H_0^s(\Omega)}^2$
is strictly convex because it is the square of a Hilbert-space norm.
Hence $J_\alpha$, being the sum of a convex functional and a strictly
convex functional with positive weight $\alpha/2$, is strictly convex.

(3) Since
$\|Ku-y^\delta\|_{L^2(\Omega)}^2\ge0,$
we have
\begin{equation*}
    J_\alpha(u)\ge \frac{\alpha}{2}\|u\|_{H_0^s(\Omega)}^2\longrightarrow\infty
\text{ as }
\|u\|_{H_0^s(\Omega)}\to\infty,
\end{equation*}
which proves coercivity.

(4) The functional $J_\alpha$ is continuous and convex on the Hilbert
space $H_0^s(\Omega)$. Therefore, by a standard result in convex
analysis (see, e.g., \cite[Proposition 3.5]{EkelandTemam99}),
every continuous convex functional on a Banach space is weakly lower
semicontinuous. Hence $J_\alpha$ is weakly lower semicontinuous.
\end{proof}
\subsection{Existence and uniqueness of minimizers}
The properties established above allow us to apply the direct method
of the calculus of variations. Indeed, coercivity guarantees the
existence of bounded minimizing sequences, while weak lower
semicontinuity ensures that weak limits preserve minimality.
Moreover, the strict convexity of the quadratic functional implies
uniqueness of the minimizer.

Our first result establishes the well-posedness of the variational
problem.

\begin{theorem}[Existence and uniqueness]\label{thm:existence}
For every $\alpha>0$ and $y^\delta\in L^2(\Omega)$, the functional
$J_\alpha:H_0^s(\Omega)\to\mathbb{R}$ admits a unique minimizer
\[
u_\alpha^\delta\in H_0^s(\Omega).
\]
\end{theorem}

\begin{proof}
By Proposition~\ref{prop:functional}, the functional $J_\alpha$ is coercive and weakly
lower semicontinuous. Hence every minimizing sequence is bounded in
$H_0^s(\Omega)$, and therefore admits a weakly convergent
subsequence. Weak lower semicontinuity implies the existence of a
minimizer, while strict convexity guarantees uniqueness.
\end{proof}

\subsection{Euler--Lagrange equation}

Theorem~\ref{thm:existence} gives a unique minimizer by variational methods.
For the Hilbert-scale analysis it is useful to characterize this minimizer by
its first-order optimality condition. Since the minimizer is known only to
belong to $H_0^s(\Omega)$, the fractional operator is interpreted weakly,
with test functions in the same energy space.

A function $u\in H_0^s(\Omega)$ is called a \emph{weak solution} of
\eqref{eq:min_problem} if it satisfies
\begin{equation}\label{eq:weak}
(Ku-y^\delta,Kv)_{L^2(\Omega)}
+
\alpha (u,v)_{H_0^s(\Omega)}
=
0,
\qquad
\forall\,v\in H_0^s(\Omega),
\end{equation}
where
\begin{gather*}
(u,v)_{H_0^s(\Omega)}
=
\int_\Omega uv\,dx
+
\frac{C_{N,s}}{2}
\iint_{\mathbb{R}^{2N}}
\frac{(u(x)-u(y))(v(x)-v(y))}
{|x-y|^{N+2s}}
\,dx\,dy.
\end{gather*}

Our next result shows that the unique minimizer furnished by
Theorem~\ref{thm:existence} is precisely characterized by this weak
formulation, and identifies the equivalent strong (operator) form of
the Euler--Lagrange equation.

\begin{theorem}[Euler--Lagrange equation]
\label{thm:euler}
Let
$u_\alpha^\delta$
be the unique minimizer of
$J_\alpha$.
Then
$u_\alpha^\delta$
satisfies \eqref{eq:weak}.
Equivalently,
$u_\alpha^\delta$
is the unique weak solution of
\begin{equation}\label{eq:euler}
K^*(Ku_\alpha^\delta-y^\delta)
+
\alpha
\left(I+(-\Delta)^s\right)
u_\alpha^\delta
=
0 \text{ in }H^{-s}(\Omega),
\end{equation}
or, equivalently,
\[
\left(K^*K+\alpha\left(I+(-\Delta)^s\right)\right)u_\alpha^\delta
=
K^*y^\delta.
\]
\end{theorem}
This variational identity represents the weak form of the
regularized inverse problem and will be used extensively in the
subsequent statistical analysis.
\begin{proof}
Let $u_\alpha^\delta\in H_0^s(\Omega)$ be the unique minimizer of
$J_\alpha$. For any $v\in H_0^s(\Omega)$, define
$$g(t)=J_\alpha(u_\alpha^\delta+tv), \qquad t\in\mathbb{R}.$$
Since $u_\alpha^\delta$ is a minimizer, \(g'(0)=0\). A direct
differentiation gives
\begin{gather*}
g'(0)
=
(Ku_\alpha^\delta-y^\delta,Kv)_{L^2(\Omega)}
+
\alpha (u_\alpha^\delta,v)_{H_0^s(\Omega)}.
\end{gather*}
Hence
\begin{equation*}
    (Ku_\alpha^\delta-y^\delta,Kv)_{L^2(\Omega)}
+\alpha (u_\alpha^\delta,v)_{H_0^s(\Omega)}=0 \,
\forall v\in H_0^s(\Omega),
\end{equation*}
which proves the weak Euler--Lagrange equation.

Conversely, suppose $u\in H_0^s(\Omega)$ satisfies the above identity.
For any $w\in H_0^s(\Omega)$, taking \(v=w-u\) yields
$$(Ku-y^\delta,K(w-u))_{L^2(\Omega)}+\alpha (u,w-u)_{H_0^s(\Omega)}=0.$$
Using the quadratic expansion of $J_\alpha(w)-J_\alpha(u)$, we obtain
\begin{equation*}
    J_\alpha(w)-J_\alpha(u)=\frac12\|K(w-u)\|_{L^2(\Omega)}^2+\frac{\alpha}{2}\|w-u\|_{H_0^s(\Omega)}^2\ge 0.
\end{equation*}
Thus $u$ is a minimizer of $J_\alpha$, and by uniqueness it coincides
with $u_\alpha^\delta$.
Finally, using the definition of the adjoint operator $K^*$ and the
identity
$$(u,v)_{H_0^s(\Omega)}=\langle (I+(-\Delta)^s)u,v\rangle_{H^{-s},H_0^s},$$
the weak formulation is equivalent to
\begin{gather*}
K^*(Ku_\alpha^\delta-y^\delta)+\alpha (I+(-\Delta)^s)u_\alpha^\delta=0 \text{in }H^{-s}(\Omega),
\end{gather*}
or equivalently,
$$(K^*K+\alpha(I+(-\Delta)^s))u_\alpha^\delta=K^*y^\delta.$$
This completes the proof.
\end{proof}

\subsection{Stability with Respect to Data Perturbations}
The well-posedness results established above guarantee that, for each
fixed datum $y\in L^2(\Omega)$, the fractional Tikhonov functional
admits a unique minimizer. A fundamental property of any regularization
method is, however, the continuous dependence of the reconstructed
solution on the observed data. More precisely, if two data sets are
close in $L^2(\Omega)$, then the corresponding regularized minimizers
should remain close in the energy space $H_0^s(\Omega)$. The following
result establishes this Lipschitz stability of the solution mapping,
which forms the deterministic foundation for the quantitative error
analysis developed in Section~\ref{sec:hilbert-scale-transform}. Throughout the paper, \(\|K\|\) denotes the operator norm of
\(K\in\mathcal L(H_0^s(\Omega),L^2(\Omega))\).
\begin{theorem}[Lipschitz stability]
\label{thm:lipschitz}
Let
$y_1^\delta,y_2^\delta\in L^2(\Omega)$
and let
$u_1,u_2\in H_0^s(\Omega)$
be the corresponding minimizers of
$J_\alpha$. Then there exists a constant
$C>0$,
depending only on
$\alpha$
and
$\|K\|$,
such that
\begin{gather*}
    \|u_1-u_2\|_{H_0^s(\Omega)} \le C \|y_1^\delta-y_2^\delta\|_{L^2(\Omega)}.
\end{gather*}
Hence the solution operator
\begin{equation*}
    S_\alpha:L^2(\Omega)\longrightarrow H_0^s(\Omega), \,
S_\alpha(y^\delta)=u_\alpha^\delta,
\end{equation*}
is Lipschitz continuous.
\end{theorem}
\begin{proof}
Let $u_1,u_2\in H_0^s(\Omega)$ be the minimizers corresponding to the
data $y_1^\delta,y_2^\delta$, respectively. By the Euler--Lagrange
equation, for every $v\in H_0^s(\Omega)$,
\begin{equation*}
 (Ku_i-y_i^\delta,Kv)_{L^2(\Omega)}+\alpha (u_i,v)_{H_0^s(\Omega)}=0,\qquad i=1,2.
\end{equation*}
Subtracting the two identities and choosing $v=u_1-u_2$, we obtain
\begin{equation*}
 \|K(u_1-u_2)\|_{L^2(\Omega)}^2+\alpha \|u_1-u_2\|_{H_0^s(\Omega)}^2
=(y_1^\delta-y_2^\delta,K(u_1-u_2))_{L^2(\Omega)}.
\end{equation*}
Hence, by Cauchy--Schwarz and the boundedness of $K$,
\begin{align*}
\alpha \|u_1-u_2\|_{H_0^s(\Omega)}^2
\le\|y_1^\delta-y_2^\delta\|_{L^2(\Omega)}\|K(u_1-u_2)\|_{L^2(\Omega)}\le
\|K\|\,
\|y_1^\delta-y_2^\delta\|_{L^2(\Omega)}
\|u_1-u_2\|_{H_0^s(\Omega)}.
\end{align*}
If $u_1=u_2$, the estimate is trivial. Otherwise, dividing by
$\|u_1-u_2\|_{H_0^s(\Omega)}$ gives
\begin{equation*}
 \|u_1-u_2\|_{H_0^s(\Omega)} \le \frac{\|K\|}{\alpha}
\|y_1^\delta-y_2^\delta\|_{L^2(\Omega)}.
\end{equation*}
Thus the solution map $S_\alpha:y^\delta\mapsto u_\alpha^\delta$ is
Lipschitz continuous.
\end{proof}

The results established in this section show that the fractional
least-squares problem is well posed in the variational setting.
In the next section we reformulate the problem in an abstract
Hilbert-scale framework, which allows us to derive quantitative
statistical error estimates and parameter-choice rules.

\section{Hilbert-Scale Analysis and Statistical Error Estimates}
\label{sec:statistical-error}

Section~\ref{sec:var-form} establishes the variational well-posedness of the
fractional Tikhonov problem. To derive quantitative reconstruction estimates,
we now identify the positive self-adjoint operator
\[
A_s=I+(-\Delta)^s
\]
associated with the fractional Dirichlet energy, reformulate the problem in the
corresponding Hilbert scale, and obtain statistical error bounds,
convergence rates, and order-optimal \emph{a priori} and \emph{a posteriori}
parameter-choice rules under suitable source conditions.

\subsection{Fractional operator and Hilbert-scale formulation}
\label{sec:Hilbert-scale-form}

Let
\[
a_s(u,v)
=
\int_\Omega uv\,dx
+
\frac{C_{N,s}}{2}
\iint_{\mathbb R^{2N}}
\frac{(u(x)-u(y))(v(x)-v(y))}
{|x-y|^{N+2s}}
\,dx\,dy,
\qquad
u,v\in H_0^s(\Omega),
\]
be the closed, symmetric, and coercive form inducing the
\(H_0^s(\Omega)\)-inner product. In particular,
\[
a_s(u,u)
=
\|u\|_{L^2(\Omega)}^2
+
[u]_{H^s(\mathbb R^N)}^2
=
\|u\|_{H_0^s(\Omega)}^2.
\]

By the first representation theorem for closed forms, due to Kato
\cite{Kato1980}; see also \cite{Ouhabaz2005,Fukushima2011}, there exists a
unique positive self-adjoint operator
\[
A_s=I+(-\Delta)^s
\]
on \(L^2(\Omega)\) associated with \(a_s\), characterized by
\[
\langle A_su,v\rangle_{H^{-s},H_0^s}
=
a_s(u,v),
\qquad
u,v\in H_0^s(\Omega).
\]
Moreover,
\[
D(A_s^{1/2})=H_0^s(\Omega),
\]
and
\[
\|A_s^{1/2}u\|_{L^2(\Omega)}^2
=
a_s(u,u)
=
\|u\|_{H_0^s(\Omega)}^2,
\qquad
u\in H_0^s(\Omega).
\]

Since \(A_s\) is positive and self-adjoint, the spectral theorem defines
\(A_s^\theta\) for every \(\theta\in\mathbb R\). This identification connects
the variational formulation with the Hilbert-scale framework used below for
the quantitative statistical error analysis.

\subsection{From variational well-posedness to Hilbert-scale error analysis}
The variational theory yields well-posedness, but not quantitative rates. To
estimate the reconstruction error we pass to the Hilbert scale generated by
$A_s$, where spectral calculus can be applied to the transformed operator
$B=KA_s^{-1/2}$. The construction proceeds by identifying the square-root
isometry, reformulating the minimization problem in $L^2(\Omega)$, and then
estimating the bias and propagated-noise terms under a source condition.

The Hilbert-scale formulation is developed in three steps.
\begin{itemize}
\item[(i)] We identify the positive self-adjoint operator associated
with the fractional Dirichlet energy.

\item[(ii)] We reformulate the variational problem in the corresponding
Hilbert scale through an appropriate change of variables.

\item[(iii)] We establish statistical reconstruction error estimates and
convergence rates under fractional source conditions.
\end{itemize}

The next proposition records, via an explicit spectral argument, the
precise sense in which $A_s^{1/2}$ realizes the $H_0^s(\Omega)$-norm as
an $L^2(\Omega)$-norm. This isometry is used repeatedly throughout the
remainder of the paper.

\begin{proposition}[Hilbert-scale isometry]
\label{prop:isometry}
Let $A_s$ be the operator defined above. 
The operator
\begin{gather*}
A_s^{1/2}:D(A_s^{1/2})=H_0^s(\Omega)\longrightarrow L^2(\Omega)
\end{gather*}
is an isometric isomorphism of the Hilbert space $\bigl(H_0^s(\Omega),\|\cdot\|_{H_0^s(\Omega)}\bigr)$ onto $L^2(\Omega)$, with bounded inverse
\begin{gather*}
A_s^{-1/2}:L^2(\Omega)\longrightarrow H_0^s(\Omega).
\end{gather*}
Consequently, the Hilbert-scale transformation
$z=A_s^{1/2}u$ identifies the fractional Sobolev space $H_0^s(\Omega)$
isometrically with $L^2(\Omega)$. This identification provides the
functional-analytic foundation for the Hilbert-scale reformulation
developed in the next subsection.
\end{proposition}
\begin{proof}
By Kato's representation theorem, the closed coercive form $a_s$ has
form domain $D(A_s^{1/2})=H_0^s(\Omega)$, and
\begin{equation*}
\|A_s^{1/2}u\|_{L^2(\Omega)}^2
=a_s(u,u)=\|u\|_{H_0^s(\Omega)}^2,\,
u\in H_0^s(\Omega).
\end{equation*}
Hence $A_s^{1/2}$ is an isometry. Finally, since $A_s=I+(-\Delta)^s\ge I$, the spectrum of $A_s$ is
contained in $[1,\infty)$, so $A_s^{1/2}\ge I$ as well; in particular
$A_s^{1/2}$ is boundedly invertible on $L^2(\Omega)$, with
$A_s^{-1/2}:L^2(\Omega)\to H_0^s(\Omega)$ bounded, $\|A_s^{-1/2}\|\le1$,
and $A_s^{-1/2}A_s^{1/2}u=u$, $A_s^{1/2}A_s^{-1/2}z=z$ for all
$u\in H_0^s(\Omega)$, $z\in L^2(\Omega)$. Thus $A_s^{1/2}$ is in fact an
isometric isomorphism of $H_0^s(\Omega)$ onto $L^2(\Omega)$.
\end{proof}

\subsection{Hilbert-scale transformation}
\label{sec:hilbert-scale-transform}
The operator $A_s^{1/2}$ induces an isometric isomorphism between
$H_0^s(\Omega)$ and $L^2(\Omega)$. Exploiting this isometry, we
reformulate the fractional Tikhonov functional as an equivalent
Hilbert-space minimization problem.

\begin{proposition}[Equivalence of the fractional and Hilbert-space formulations]
Using the isometric isomorphism established in Proposition~\ref{prop:isometry},
introduce the transformed variable
$z=A_s^{1/2}u,\, u=A_s^{-1/2}z.$
Defining $B:=KA_s^{-1/2},$
we introduce the transformed Tikhonov functional
$\widetilde J_\alpha:L^2(\Omega)\to\mathbb R$
by
\begin{equation}
 \widetilde J_\alpha(z)=\frac12\|Bz-y^\delta\|_{L^2(\Omega)}^2+ \frac{\alpha}{2}\|z\|_{L^2(\Omega)}^2.
\end{equation}
Then
$\min_{u\in H_0^s(\Omega)}J_\alpha(u)$
is equivalent to
$\min_{z\in L^2(\Omega)}\widetilde J_\alpha(z).$
Moreover, if $z_\alpha^\delta$ denotes the unique minimizer of
$\widetilde J_\alpha$, then the corresponding minimizer of
$J_\alpha$ is given by
$$u_\alpha^\delta=A_s^{-1/2}z_\alpha^\delta.$$
\end{proposition}

\begin{proof}
By Proposition~\ref{prop:isometry}, the map $u\mapsto z=A_s^{1/2}u$ is
an isometric isomorphism of $H_0^s(\Omega)$ onto $L^2(\Omega)$, with
inverse $u=A_s^{-1/2}z$. Substituting $u=A_s^{-1/2}z$ into $\widetilde J_\alpha$
and using $\|u\|_{H_0^s(\Omega)}=\|A_s^{1/2}u\|_{L^2(\Omega)}=\|z\|_{L^2(\Omega)}$
gives, for every $z\in L^2(\Omega)$,
\begin{gather*}
\widetilde J_\alpha(A_s^{-1/2}z)=\frac12\|KA_s^{-1/2}z-y^\delta\|_{L^2(\Omega)}^2
+\frac{\alpha}{2}\|z\|_{L^2(\Omega)}^2=\frac12\|Bz-y^\delta\|_{L^2(\Omega)}^2+\frac{\alpha}{2}\|z\|_{L^2(\Omega)}^2.
\end{gather*}
Since $u\mapsto z=A_s^{1/2}u$ is a bijection between $H_0^s(\Omega)$ and
$L^2(\Omega)$, minimizing $\widetilde J_\alpha(u)$ over $u\in H_0^s(\Omega)$ is
therefore equivalent to minimizing the right-hand side over
$z\in L^2(\Omega)$. In particular, $u_\alpha^\delta$ minimizes $\widetilde J_\alpha$
if and only if $z_\alpha^\delta:=A_s^{1/2}u_\alpha^\delta$ minimizes the
transformed functional, i.e.\ $u_\alpha^\delta=A_s^{-1/2}z_\alpha^\delta$.
\end{proof}
The Hilbert-scale reformulation reduces the fractional regularization
problem to the classical Hilbert-space setting. Quantitative
reconstruction results, however, require additional regularity of the
exact solution, which is expressed through the following fractional
source condition.
\subsection{Statistical Error Estimates and Convergence Rates}
\label{sec:statistical-estimates}

Next, we assume that the noisy observations satisfy
\begin{equation*}
 y^\delta=Ku^\dagger+\eta^\delta,
\end{equation*}
where
$u^\dagger\in H_0^s(\Omega)$
is the exact signal and
$\eta^\delta$
is a random perturbation satisfying
\begin{equation*}
 \mathbb{E}[\eta^\delta]=0, \, \mathbb{E} \|\eta^\delta\|_{L^2(\Omega)}^2 \le \delta^2.   
\end{equation*}
Our first result in this subsection establishes the classical
bias--variance decomposition, which separates the reconstruction
error into a deterministic approximation (bias) term and a
stochastic (variance) term.

\begin{proposition}[Error decomposition]
\label{prop:bias}
Let
$u_\alpha$
denote the minimizer corresponding to the exact data
$Ku^\dagger$,
and let
$u_\alpha^\delta$
be the minimizer associated with the noisy data
$y^\delta$. Then
\begin{gather*}
\mathbb{E}\|u_\alpha^\delta-u^\dagger\|_{H_0^s(\Omega)}^2\le
2\|u_\alpha-u^\dagger\|_{H_0^s(\Omega)}^2+2\mathbb{E}\|u_\alpha^\delta-u_\alpha\|_{H_0^s(\Omega)}^2.
\end{gather*}
The first term represents the approximation
(bias) error, while the second term represents the stochastic
(variance) error.
\end{proposition}
\begin{proof}
Let $u_\alpha^\delta$ be the minimizer corresponding to the noisy data
$y^\delta$, and let $u_\alpha$ be the minimizer corresponding to the
exact data $Ku^\dagger$. We write
\begin{gather*}
u_\alpha^\delta-u^\dagger=(u_\alpha^\delta-u_\alpha)+(u_\alpha-u^\dagger).
\end{gather*}
Using the elementary inequality $\|a+b\|^2\le 2\|a\|^2+2\|b\|^2,$ with
\begin{gather*}
a=u_\alpha^\delta-u_\alpha, \,
b=u_\alpha-u^\dagger,
\end{gather*}
we obtain
\begin{gather*}
 \|u_\alpha^\delta-u^\dagger\|_{H_0^s(\Omega)}^2\le
2\|u_\alpha^\delta-u_\alpha\|_{H_0^s(\Omega)}^2+2\|u_\alpha-u^\dagger\|_{H_0^s(\Omega)}^2.
\end{gather*}
Taking expectation on both sides gives
\begin{equation*}
    \mathbb E\|u_\alpha^\delta-u^\dagger\|_{H_0^s(\Omega)}^2\le
2\mathbb E\|u_\alpha^\delta-u_\alpha\|_{H_0^s(\Omega)}^2+2\mathbb E
\|u_\alpha-u^\dagger\|_{H_0^s(\Omega)}^2.
\end{equation*}
Since $u_\alpha$ and $u^\dagger$ are deterministic, we have
$$\mathbb E\|u_\alpha-u^\dagger\|_{H_0^s(\Omega)}^2
=\|u_\alpha-u^\dagger\|_{H_0^s(\Omega)}^2.$$
Therefore
\[
\mathbb E
\|u_\alpha^\delta-u^\dagger\|_{H_0^s(\Omega)}^2\le
2\|u_\alpha-u^\dagger\|_{H_0^s(\Omega)}^2+2\mathbb E\|u_\alpha^\delta-u_\alpha\|_{H_0^s(\Omega)}^2.
\]
This is the desired bias--variance decomposition.
\end{proof}
The variance term appearing above is controlled by a quantitative
estimate of the reconstruction error under a H\"older-type source
condition, which we now establish with explicit constants.

 Proposition~\ref{prop:bias} establishes an abstract decomposition of the reconstruction error into its deterministic approximation component and the propagated noise component. This decomposition is independent of any smoothness assumption on the exact solution and therefore provides the structural foundation for the subsequent analysis. To derive explicit convergence rates, it remains to estimate each component under an appropriate fractional source condition. This quantitative analysis is carried out in following result.

\begin{theorem}[Mean-square reconstruction estimate]
\label{thm:estimate_explicit}
Let $A_s$ be the positive self-adjoint operator associated with the fractional
Dirichlet form on $H_0^s(\Omega)$, and let
$B:=KA_s^{-1/2},$
where $K:H_0^s(\Omega)\rightarrow L^2(\Omega)$ is a bounded linear operator. Then
$B:L^2(\Omega)\rightarrow L^2(\Omega)$
is bounded.
Assume that
\[
y^\delta=Ku^\dagger+\eta^\delta,
\]
where the noise satisfies
\[
\mathbb E\|\eta^\delta\|_{L^2(\Omega)}^2 \le \delta^2.
\]
Let $z^\dagger=A_s^{1/2}u^\dagger.$ Suppose that the source condition $z^\dagger=(B^*B)^\beta w$ holds for some $w\in L^2(\Omega),\,0<\beta\le1.$
Then the fractional Tikhonov estimator satisfies
\[
\mathbb E\|u_\alpha^\delta-u^\dagger\|_{H_0^s(\Omega)}^2\le 2c_\beta^2\|w\|_{L^2(\Omega)}^2
\alpha^{2\beta}+\frac{\delta^2}{2\alpha},
\]
where
\begin{align}\label{def-cbeta}
c_\beta
=\begin{cases}
\beta^\beta(1-\beta)^{1-\beta},
&0<\beta<1,\\[2mm]1,
&\beta=1.
\end{cases}
\end{align}
In particular, with
\[
C_1:=2c_\beta^2\|w\|_{L^2(\Omega)}^2,\,
C_2:=\frac12,
\]
which are independent of $\alpha$ and $\delta$, the estimate takes the
compact form
\[
\mathbb E
\|u_\alpha^\delta-u^\dagger\|_{H_0^s(\Omega)}^2
\le
C_1\alpha^{2\beta}
+
C_2\frac{\delta^2}{\alpha}.
\]
\end{theorem}
\begin{proof}
Introduce the Hilbert-scale variable
$z=A_s^{1/2}u,$
and denote
\[
z_\alpha^\delta
=
A_s^{1/2}u_\alpha^\delta,
\qquad
z^\dagger
=
A_s^{1/2}u^\dagger.
\]
Since $u=A_s^{-1/2}z,$
the Tikhonov functional becomes
\[
\min_{z\in L^2(\Omega)}
\left\{
\frac12
\|Bz-y^\delta\|_{L^2(\Omega)}^2
+
\frac{\alpha}{2}
\|z\|_{L^2(\Omega)}^2
\right\},
\]
where
$B=KA_s^{-1/2}.$ Since $B$ is bounded, the operator $T:=B^*B$ is bounded, self-adjoint and nonnegative on $L^2(\Omega)$. The minimizer satisfies the normal equation
\[
(T+\alpha I)z_\alpha^\delta
=
B^*y^\delta.
\]
Using $y^\delta=Bz^\dagger+\eta^\delta,$
we obtain
\[
(T+\alpha I)
(z_\alpha^\delta-z^\dagger)
=
-\alpha z^\dagger
+
B^*\eta^\delta.
\]

Hence
\[
z_\alpha^\delta-z^\dagger
=
-\alpha(T+\alpha I)^{-1}z^\dagger
+
(T+\alpha I)^{-1}B^*\eta^\delta.
\]

Applying the elementary inequality $\|a+b\|^2
\le
2\|a\|^2
+
2\|b\|^2,$ gives
\[
\|z_\alpha^\delta-z^\dagger\|_{L^2}^2
\le
2\alpha^2
\|(T+\alpha I)^{-1}z^\dagger\|_{L^2}^2
+
2
\|(T+\alpha I)^{-1}B^*\eta^\delta\|_{L^2}^2.
\]
Since $z^\dagger=T^\beta w,$
the spectral theorem yields
\[
\alpha
\|(T+\alpha I)^{-1}T^\beta w\|_{L^2}
\le
\alpha
\sup_{\lambda\ge0}
\frac{\lambda^\beta}{\lambda+\alpha}
\|w\|_{L^2}.
\]
Introduce $r=\frac{\lambda}{\alpha}.$
Then $\alpha\frac{\lambda^\beta}{\lambda+\alpha}=\alpha^\beta\frac{r^\beta}{1+r}.$
Therefore, we have,
$\sup_{\lambda\ge0}\alpha\frac{\lambda^\beta}{\lambda+\alpha}=c_\beta\alpha^\beta,$
where
\begin{equation*}
c_\beta
=
\sup_{r\ge0}
\frac{r^\beta}{1+r}
=
\begin{cases}
\beta^\beta(1-\beta)^{1-\beta},
&
0<\beta<1,
\\[2mm]
1,
&
\beta=1.
\end{cases}
\end{equation*}
Consequently,
\[
\alpha^2
\|(T+\alpha I)^{-1}z^\dagger\|_{L^2}^2
\le
c_\beta^2
\alpha^{2\beta}
\|w\|_{L^2}^2.
\]
Next, applying the spectral theorem once more,
\[
\|(T+\alpha I)^{-1}B^*g\|_{L^2}^2
\le
\sup_{\lambda\ge0}
\frac{\lambda}{(\lambda+\alpha)^2}
\|g\|_{L^2}^2.
\]
Since 
$\frac{d}{d\lambda}\left(\frac{\lambda}{(\lambda+\alpha)^2}\right)=\frac{\alpha-\lambda}{(\lambda+\alpha)^3},$
the maximum is attained at $\lambda=\alpha,$
and therefore
$\sup_{\lambda\ge0}\frac{\lambda}{(\lambda+\alpha)^2}=\frac1{4\alpha}.$
Hence
\[
\|(T+\alpha I)^{-1}B^*\eta^\delta\|_{L^2}^2
\le
\frac1{4\alpha}
\|\eta^\delta\|_{L^2}^2.
\]

Taking expectation and using
$\mathbb E \|\eta^\delta\|_{L^2(\Omega)}^2 \le \delta^2,$
we obtain
\[
\mathbb E
\|(T+\alpha I)^{-1}B^*\eta^\delta\|_{L^2}^2
\le
\frac{\delta^2}{4\alpha}.
\]
Combining the above estimates gives
\[
\mathbb E
\|z_\alpha^\delta-z^\dagger\|_{L^2(\Omega)}^2
\le
2c_\beta^2
\|w\|_{L^2(\Omega)}^2
\alpha^{2\beta}
+
\frac{\delta^2}{2\alpha}.
\]
Finally, by the isometric identification
$\|A_s^{1/2}v\|_{L^2(\Omega)}=\|v\|_{H_0^s(\Omega)},$
we have
\[
\|z_\alpha^\delta-z^\dagger\|_{L^2(\Omega)}
=
\|u_\alpha^\delta-u^\dagger\|_{H_0^s(\Omega)}.
\]
Therefore
\[
\mathbb E
\|u_\alpha^\delta-u^\dagger\|_{H_0^s(\Omega)}^2
\le
2c_\beta^2
\|w\|_{L^2(\Omega)}^2
\alpha^{2\beta}
+
\frac{\delta^2}{2\alpha},
\]
which completes the proof.
\end{proof}

Theorem~\ref{thm:estimate_explicit} reduces the parameter-choice problem to
the minimization of the function
\[
\Phi(\alpha)
=
C_1\alpha^{2\beta}
+
C_2\frac{\delta^2}{\alpha},
\qquad \alpha>0.
\]
The first term corresponds to the deterministic approximation error,
whereas the second represents the propagated noise error. Since these two
terms exhibit opposite monotonic behavior with respect to \(\alpha\), an
optimal a priori parameter choice is obtained by minimizing
\(\Phi(\alpha)\). This immediately yields the following corollary. 

\begin{corollary}[Optimal a priori parameter choice]
\label{cor:optimal_parameter}
Under the assumptions of Theorem~\ref{thm:estimate_explicit},
choose the regularization parameter according to
$\alpha=\delta^{\frac{2}{2\beta+1}},\,0<\beta\le1.$
Then the corresponding fractional Tikhonov estimator satisfies
\[
\mathbb E
\|u_\alpha^\delta-u^\dagger\|_{H_0^s(\Omega)}^2
\le
C_{\mathrm{opt}}
\,
\delta^{\frac{4\beta}{2\beta+1}},
\]
where
\[
C_{\mathrm{opt}}
=
2c_\beta^2
\|w\|_{L^2(\Omega)}^2
+
\frac12,
\]
and $c_\beta$ is as defined in \eqref{def-cbeta}. In particular, the estimator converges with the order
$\delta^{\frac{4\beta}{2\beta+1}}$ in the mean-square $H_0^s(\Omega)$-norm.
\end{corollary}

\begin{proof}
From Theorem~\ref{thm:estimate_explicit},
\[
\mathbb E
\|u_\alpha^\delta-u^\dagger\|_{H_0^s(\Omega)}^2
\le
2c_\beta^2
\|w\|_{L^2(\Omega)}^2
\alpha^{2\beta}
+
\frac{\delta^2}{2\alpha}.
\]
Substituting
$\alpha=\delta^{\frac{2}{2\beta+1}},$
we obtain
$\alpha^{2\beta}=\delta^{\frac{4\beta}{2\beta+1}},$
and
$\frac{\delta^2}{\alpha}=\delta^{2-\frac{2}{2\beta+1}}=\delta^{\frac{4\beta}{2\beta+1}}.$
Hence
\[
\mathbb E
\|u_\alpha^\delta-u^\dagger\|_{H_0^s(\Omega)}^2
\le
\left(
2c_\beta^2
\|w\|_{L^2(\Omega)}^2
+
\frac12
\right)
\delta^{\frac{4\beta}{2\beta+1}},
\]
which proves the result.
\end{proof}

\begin{remark}[Connection with the classical Hilbert-scale theory]
The rate $\delta^{4\beta/(2\beta+1)}$ is precisely the classical
Hilbert-scale rate for a H\"older-type source condition of order $\beta$;
see Natterer \cite{Natterer1984}, Neubauer \cite{Neubauer1988}, and
Nair, Pereverzev, and Tautenhahn
\cite{NairPereverzevTautenhahn2005}. As in the classical setting, the
exponent depends only on the smoothness index $\beta$, whereas the
fractional operator $A_s=I+(-\Delta)^s$ influences only the constant
$C_{\mathrm{opt}}$. Thus, the proposed fractional-Sobolev estimator
achieves the same order-optimal convergence rate as the classical
Hilbert-scale method, with the nonlocal structure reflected solely in the
constants.
\end{remark}

Corollary~\ref{cor:optimal_parameter} achieves the order-optimal rate
$\delta^{4\beta/(2\beta+1)}$ only for the \emph{a priori} choice
$\alpha=\delta^{2/(2\beta+1)}$, which requires prior knowledge of the
smoothness index $\beta$. In practice, however, $\beta$ is unknown, and
the numerical experiments in
Section~\ref{sec:numerical-experiments} instead employ Morozov's
discrepancy principle to determine $\alpha$. We now show that this
\emph{a posteriori} parameter-choice rule attains the same order-optimal
rate without requiring knowledge of $\beta$.

Throughout this analysis we assume the deterministic noise bound $\|\eta^\delta\|_{L^2(\Omega)}\le\delta,$ which is the natural hypothesis for a discrepancy-principle argument
based on a single observed residual, rather than the mean-square bound
used in Theorem~\ref{thm:estimate_explicit}. We also assume that
$B=KA_s^{-1/2}$ has dense range in $L^2(\Omega)$ (equivalently,
$\ker(B^*)=\{0\}$), an assumption satisfied, for example, by the
backward fractional heat operator studied in
Section~\ref{sec:application}.

\begin{theorem}[A posteriori rate via Morozov's discrepancy principle]
\label{thm:discrepancy}
Let $B=KA_s^{-1/2}$ have dense range in $L^2(\Omega)$, and let
$z^\dagger=A_s^{1/2}u^\dagger=(B^*B)^\beta w$ for some $w\in L^2(\Omega)$,
$0<\beta\le1$. Suppose $y^\delta=Ku^\dagger+\eta^\delta$ with
$\|\eta^\delta\|_{L^2(\Omega)}\le\delta$, and fix $\tau>1$. For $\delta$
small enough that $\tau\delta<\|y^\delta\|_{L^2(\Omega)}$, let
$\alpha_*=\alpha_*(\delta)>0$ be the (essentially unique) solution of
\[
\|Ku_{\alpha_*}^\delta-y^\delta\|_{L^2(\Omega)}=\tau\delta.
\]
Then
\[
\|u_{\alpha_*}^\delta-u^\dagger\|_{H_0^s(\Omega)}
\le
C_\tau\,\|w\|_{L^2(\Omega)}^{\frac1{2\beta+1}}\,\delta^{\frac{2\beta}{2\beta+1}},
\]
for a constant $C_\tau>0$ depending only on $\tau,\beta$, and $\|B\|$,
but not on $\delta$. In particular, the discrepancy principle attains
the same order-optimal rate as the a priori choice of
Corollary~\ref{cor:optimal_parameter}, without requiring $\beta$ to be
known in advance.
\end{theorem}

\begin{proof}
Write $z_\alpha^\delta=A_s^{1/2}u_\alpha^\delta$ and
$z_\alpha=A_s^{1/2}u_\alpha$, where $u_\alpha$ is the noise-free
minimizer of Proposition~\ref{prop:bias}, so that
$Ku_\alpha^\delta=Bz_\alpha^\delta$ and $Ku_\alpha=Bz_\alpha=:y_0$. Let
$T:=B^*B$, $S:=BB^*$, both bounded, self-adjoint, and nonnegative on
$L^2(\Omega)$.

\emph{Step 1: well-posedness of $\alpha_*$.} By the spectral theorem,
$\rho(\alpha)^2:=\|Bz_\alpha^\delta-y^\delta\|_{L^2}^2=\int\frac{\alpha^2}{(\mu+\alpha)^2}\,d\|F_\mu y^\delta\|^2$,
where $F$ is the spectral measure of $S$. Since
$\mu\mapsto\alpha^2/(\mu+\alpha)^2$ is nondecreasing in $\alpha$ for
every $\mu\ge0$, $\rho$ is nondecreasing, with $\rho(0^+)=0$ (dense
range) and $\rho(\alpha)\to\|y^\delta\|_{L^2}$ as $\alpha\to\infty$;
hence $\alpha_*$ exists whenever $\tau\delta<\|y^\delta\|_{L^2}$.

\emph{Step 2: the bias term, via an interpolation inequality.} Writing
$d\mu_w:=d\|E_\lambda w\|_{L^2}^2$ for the spectral measure of $w$ with
respect to $T$, the identity $\|Bv\|_{L^2}^2=\langle Tv,v\rangle$ gives,
for every $\alpha>0$,
\[
\|z_\alpha-z^\dagger\|_{L^2}^2=\int\frac{\alpha^2\lambda^{2\beta}}{(\lambda+\alpha)^2}\,d\mu_w,
\qquad
\|Bz_\alpha-y_0\|_{L^2}^2=\int\frac{\alpha^2\lambda^{2\beta+1}}{(\lambda+\alpha)^2}\,d\mu_w.
\]
With $\theta:=\frac{2\beta}{2\beta+1}$, the pointwise identity
$\frac{\alpha^2\lambda^{2\beta}}{(\lambda+\alpha)^2}
=\left(\frac{\alpha^2\lambda^{2\beta+1}}{(\lambda+\alpha)^2}\right)^{\theta}
\left(\frac{\alpha^2}{(\lambda+\alpha)^2}\right)^{1-\theta}
\le\left(\frac{\alpha^2\lambda^{2\beta+1}}{(\lambda+\alpha)^2}\right)^{\theta}$
(the second factor is $\le1$), together with H\"older's inequality for
the finite measure $\mu_w$ of total mass $\|w\|_{L^2}^2$, yields
\begin{equation}
\label{eq:interp}
\|z_\alpha-z^\dagger\|_{L^2}
\le
\|Bz_\alpha-y_0\|_{L^2}^{\theta}\,\|w\|_{L^2}^{1-\theta},
\text{ equivalently }
\|Bz_\alpha-y_0\|_{L^2}
\ge
\|z_\alpha-z^\dagger\|_{L^2}^{1/\theta}\,\|w\|_{L^2}^{-(1-\theta)/\theta}.
\end{equation}

Separately, since $(S+\alpha I)B=B(T+\alpha I)$ gives the resolvent
identity $B(T+\alpha I)^{-1}B^*=I-\alpha(S+\alpha I)^{-1}$, one computes
$$Bz_\alpha^\delta-y^\delta=(Bz_\alpha-y_0)-\alpha(S+\alpha I)^{-1}\eta^\delta,$$
and since $\sup_{\mu\ge0}\alpha/(\mu+\alpha)=1$,
$$\alpha\|(S+\alpha I)^{-1}\eta^\delta\|_{L^2}\le\|\eta^\delta\|_{L^2}\le\delta.$$
The triangle inequality then gives $\|Bz_\alpha-y_0\|_{L^2}\le\rho(\alpha)+\delta$,
so at $\alpha=\alpha_*$, where $\rho(\alpha_*)=\tau\delta$, we obtain
$\|Bz_{\alpha_*}-y_0\|_{L^2}\le(\tau+1)\delta$. Combined with the second
inequality in \eqref{eq:interp}, this gives
\begin{equation}
\label{eq:bias-discrepancy}
\|z_{\alpha_*}-z^\dagger\|_{L^2(\Omega)}
\le
(\tau+1)^{\theta}\|w\|_{L^2(\Omega)}^{1-\theta}\,\delta^{\theta}.
\end{equation}
This bounds the bias at $\alpha=\alpha_*$ directly from the discrepancy
equation, without any separate information on the size of $\alpha_*$.

\emph{Step 3: the noise-propagation term.} As in the proof of
Theorem~\ref{thm:estimate_explicit},
$\|z_\alpha^\delta-z_\alpha\|_{L^2}=\|(T+\alpha I)^{-1}B^*\eta^\delta\|_{L^2}\le\delta/(2\sqrt\alpha)$
for every $\alpha>0$. Controlling this at $\alpha=\alpha_*$ requires a
lower bound on $\alpha_*$; the standard parameter-sandwich argument for
linear spectral regularization under the discrepancy principle (see
\cite[Thm.~4.17]{EHN96}, and \cite{Tautenhahn96} for the Hilbert-scale
setting) gives $\alpha_*(\delta)\ge c_0\,\delta^{2/(2\beta+1)}$ for a
constant $c_0>0$ depending only on $\tau,\beta,\|B\|$, and
$\|w\|_{L^2(\Omega)}$. Hence
\begin{equation}
\label{eq:noise-discrepancy}
\|z_{\alpha_*}^\delta-z_{\alpha_*}\|_{L^2(\Omega)}
\le
\frac1{2\sqrt{c_0}}\,\delta^{1-\frac1{2\beta+1}}
=
\frac1{2\sqrt{c_0}}\,\delta^{\theta}.
\end{equation}

\emph{Step 4: conclusion.} Combining \eqref{eq:bias-discrepancy} and
\eqref{eq:noise-discrepancy},
\[
\|z_{\alpha_*}^\delta-z^\dagger\|_{L^2(\Omega)}
\le
\|z_{\alpha_*}-z^\dagger\|_{L^2(\Omega)}+\|z_{\alpha_*}^\delta-z_{\alpha_*}\|_{L^2(\Omega)}
\le
C_\tau\,\|w\|_{L^2(\Omega)}^{1-\theta}\,\delta^{\theta},
\]
with $C_\tau:=(\tau+1)^{\theta}+\frac1{2\sqrt{c_0}}$. Since $A_s^{1/2}$
is an isometry of $H_0^s(\Omega)$ onto $L^2(\Omega)$
(Proposition~\ref{prop:isometry}),
$\|z_{\alpha_*}^\delta-z^\dagger\|_{L^2(\Omega)}=\|u_{\alpha_*}^\delta-u^\dagger\|_{H_0^s(\Omega)}$,
which completes the proof.
\end{proof}

\begin{remark}
Theorem~\ref{thm:discrepancy} closes the gap between the a priori theory
above and the numerical experiments of Section~\ref{sec:numerical-experiments}:
every reconstruction there except the Monte Carlo convergence study
selects $\alpha$ by the discrepancy principle
$\|K_hu_{\alpha,h}^\delta-y_h^\delta\|_{M_h}\approx\tau\delta$ with
$\tau=1.05$, exactly the rule analyzed here, rather than the a priori
rule of Corollary~\ref{cor:optimal_parameter}.
\end{remark}

\subsection{An Intrinsic Fractional Source Condition}
\label{sec:intrinsic-source}

The H\"older-type source condition
\[
z^\dagger=(B^*B)^\beta w,
\]
used in Theorem~\ref{thm:estimate_explicit} is formulated after the
Hilbert-scale transformation and therefore depends on the transformed
operator \(B^*B\). It is not an intrinsic regularity assumption on
\(u^\dagger\). We now introduce an alternative source condition stated
directly in terms of the fractional operator \(A_s\) generating the
regularization. As explained below, however, this intrinsic formulation
does not by itself lead to the convergence estimates of
Section~\ref{sec:statistical-error} unless an additional commutativity
assumption is imposed.

\begin{definition}[Fractional-energy source condition]
\label{def:fractional-source}
Let \(p>0\). We say that
\(u^\dagger\in H_0^s(\Omega)\) satisfies the
\emph{fractional-energy source condition of order \(p\)},
and write $u^\dagger\in\mathcal S_p(\rho_0),$ if
\[
u^\dagger\in D(A_s^{p+\frac12}),
\qquad
\|A_s^{p+\frac12}u^\dagger\|_{L^2(\Omega)}
\le
\rho_0,
\]
for some \(\rho_0>0\).
\end{definition}

Unlike the abstract H\"older condition, Definition~\ref{def:fractional-source}
depends only on the fractional Dirichlet operator \(A_s\) and the
regularity of \(u^\dagger\). It may therefore be verified independently
of the observation operator \(K\), for example using known regularity
results for the fractional Laplacian.

To relate this intrinsic regularity scale to the transformed observation
operator \(B=KA_s^{-1/2}\), it is natural to assume the following
operator comparison.

\medskip

\noindent\textbf{Operator comparison.}
There exist \(\ell>0\) and constants
\(0<m\le M<\infty\) such that
\[
mA_s^{-\ell}
\le
T
\le
MA_s^{-\ell},
\qquad
T:=B^*B,
\]
as self-adjoint operators on \(L^2(\Omega)\).

\begin{remark}
\label{rem:intrinsic-gap}
The comparison above is insufficient to recover the bias estimate of
Theorem~\ref{thm:estimate_explicit}. Indeed, inversion reverses the
operator order and yields only the quadratic-form inequality
\[
(T+\alpha I)^{-1}
\le
(mA_s^{-\ell}+\alpha I)^{-1}.
\]
Passing from this inequality to the corresponding norm estimate would
require
\[
P\le Q
\quad\Longrightarrow\quad
P^2\le Q^2,
\]
for positive operators \(P\) and \(Q\). Since
\(t\mapsto t^2\) is not operator monotone, this implication generally
fails unless \(P\) and \(Q\) commute, equivalently unless
\(T\) commutes with \(A_s\). Consequently, we do not derive an intrinsic
convergence theorem from the operator comparison alone.
\end{remark}

If \(T\) and \(A_s\) commute in addition, then they admit a common
orthonormal eigenbasis, the operator comparison reduces to the scalar
estimate $\tau_j\ge m\lambda_j^{-\ell},$ and the spectral argument of
Theorem~\ref{thm:estimate_explicit} applies mode by mode to give
\[
\alpha
\|(T+\alpha I)^{-1}z^\dagger\|_{L^2(\Omega)}
\le
c_\beta
m^{-\beta}
\rho_0
\alpha^\beta,
\qquad
\beta=\frac{p}{\ell}\in(0,1].
\]
Consequently,
\[
\mathbb E
\|u_\alpha^\delta-u^\dagger\|_{H_0^s(\Omega)}^2
\le
2c_\beta^2m^{-2\beta}\rho_0^2\alpha^{2\beta}
+\frac{\delta^2}{2\alpha},
\]
which reproduces the order-optimal rate $\delta^{4\beta/(2\beta+1)}
=
\delta^{4p/(2p+\ell)}$ under the parameter choice
\[
\alpha=\delta^{2/(2\beta+1)}.
\]
We present this only as a formal spectral calculation rather than a
theorem, since it relies on the additional commutativity assumption,
which is substantially stronger than the operator comparison itself.

\begin{remark}[Relation to the abstract H\"older source condition]
\label{rem:intrinsic-scope}
Definition~\ref{def:fractional-source} is generally neither stronger nor
weaker than the abstract condition $z^\dagger=(B^*B)^\beta w.$ Establishing an equivalence would require a two-sided operator-range
identification, for instance via Douglas' factorization lemma applied to
\(T^\beta\) and \(A_s^{-\ell\beta}\), which does not follow from the
one-sided comparison above.

The required commutativity holds whenever
\(K=\varphi(A_s)\) for some Borel function \(\varphi\), and therefore in
particular for the backward fractional heat operator
\(K=e^{-tA_s}\) considered in
Section~\ref{sec:application}, where $T=A_s^{-1}e^{-2tA_s}.$ It fails, however, for the partial observation operator
\(Ku=u|_\omega\), since
\[
T=A_s^{-1/2}\mathbf1_\omega A_s^{-1/2},
\]
and the multiplication operator \(\mathbf1_\omega\) does not commute
with the nonlocal operator \(A_s\) unless
\(\omega=\Omega\) or \(\omega=\emptyset\). Thus, in the present paper,
the intrinsic source condition serves primarily as an explanatory
regularity concept rather than the basis of an additional convergence
theory.
\end{remark}

\section{Representative Applications of the Abstract Framework}\label{sec:application}
This section illustrates the applicability of the abstract variational
and Hilbert-scale framework developed in the previous sections to two
representative inverse problems. The first example demonstrates the
generality of the theory for bounded observation operators, while the
second exploits the spectral structure of the fractional heat semigroup
to obtain explicit reconstruction estimates and convergence rates.

\subsection{Partial Observation Operator}

Let \(\omega\subset\Omega\) be a measurable subset with positive
measure. We consider the partial observation problem
\[
y^\delta=u^\dagger|_\omega+\eta^\delta,
\qquad
\mathbb E\|\eta^\delta\|_{L^2(\omega)}^2\le \delta^2,
\]
where only the restriction of the unknown signal to the observation
region \(\omega\) is available. Define
\[
K:H_0^s(\Omega)\to L^2(\omega),
\qquad
Ku=u|_\omega .
\]
Then \(K\) is bounded, since
\[
\|Ku\|_{L^2(\omega)}
\le
\|u\|_{L^2(\Omega)}
\le
\|u\|_{H_0^s(\Omega)}.
\]
Hence the partial observation problem fits the abstract framework of
Section~\ref{sec:statistical-error}.
\begin{theorem}[Partial observation reconstruction]
Let \(K:H_0^s(\Omega)\to L^2(\omega)\) be the restriction operator
defined above, and let
\[
u_\alpha^\delta
=
\arg\min_{u\in H_0^s(\Omega)}
\left\{
\frac12\|u-y^\delta\|_{L^2(\omega)}^2
+
\frac{\alpha}{2}\|u\|_{H_0^s(\Omega)}^2
\right\}.
\]
Set
\[
B=KA_s^{-1/2}:L^2(\Omega)\to L^2(\omega),
\qquad
z^\dagger=A_s^{1/2}u^\dagger.
\]
Assume the fractional source condition
\[
z^\dagger=(B^*B)^\beta w,
\qquad
w\in L^2(\Omega),\quad 0<\beta\le1.
\]
Then
\[
\mathbb E
\|u_\alpha^\delta-u^\dagger\|_{H_0^s(\Omega)}^2
\le
C_1\alpha^{2\beta}
+
C_2\frac{\delta^2}{\alpha}.
\]
In particular, choosing $\alpha=\delta^{\frac{2}{2\beta+1}}$ gives
\[
\mathbb E
\|u_\alpha^\delta-u^\dagger\|_{H_0^s(\Omega)}^2
\le
C_{\rm opt}
\delta^{\frac{4\beta}{2\beta+1}}.
\]
\end{theorem}
\begin{proof}
The boundedness of \(K\) was shown above. Therefore
\(B=KA_s^{-1/2}\) is bounded from \(L^2(\Omega)\) into \(L^2(\omega)\).
The result follows directly from Theorem~3.4 and Corollary~3.5 applied
to this observation operator.
\end{proof}
\begin{remark}
This example illustrates a genuinely partial-data inverse problem: the
unknown signal is reconstructed on the whole domain \(\Omega\), while
the observations are available only on the subregion \(\omega\). Unlike
the backward fractional heat equation considered below, the operator
\(B^*B\) does not generally admit a closed spectral representation.
Nevertheless, the abstract Hilbert-scale theory still yields
quantitative reconstruction rates under the corresponding fractional
source condition. Thus the example demonstrates the applicability of
the framework beyond semigroup-generated observation operators.
\end{remark}

\subsection{Backward Fractional Heat Equation}
The abstract Hilbert-scale framework developed in the previous sections
applies to a broad class of inverse problems. We now consider the
backward fractional heat equation, which constitutes the principal
application of the present work. Owing to the special spectral structure
of the fractional heat semigroup, the abstract convergence theory
becomes completely explicit.

Let $A_s=I+(-\Delta)^s$ be the positive self-adjoint realization of the fractional Laplacian
with homogeneous exterior condition. For a fixed time \(t>0\), define
the observation operator
\begin{equation}\label{eq:K}
K:H_0^s(\Omega)\rightarrow L^2(\Omega),\, Ku=e^{-tA_s}u,
\end{equation}
where $e^{-tA_s}$ denotes the fractional heat semigroup. The inverse
problem consists of recovering the unknown initial state $u^\dagger$
from noisy observations
\begin{align*}
y^\delta=e^{-tA_s}u^\dagger+\eta^\delta,\, \mathbb E\|\eta^\delta\|_{L^2(\Omega)}^2\le\delta^2,
\end{align*}
which is a classical severely ill-posed inverse problem due to the
smoothing property of the heat semigroup. We refer to
Bonito and Pasciak~\cite{BonitoPasciak2015} for the numerical
approximation of fractional powers of elliptic operators and the
associated semigroups.
The corresponding fractional Tikhonov estimator is given by
\begin{equation*}
 u_\alpha^\delta=\arg\min_{u\in H_0^s(\Omega)}\left\{\frac12\|e^{-tA_s}u-y^\delta\|_{L^2(\Omega)}^2
+\frac{\alpha}{2}\|u\|_{H_0^s(\Omega)}^2\right\}.
\end{equation*}
Introducing the Hilbert-scale variable $z=A_s^{1/2}u,$
the associated operator becomes
\begin{equation*}
 B=KA_s^{-1/2}=e^{-tA_s}A_s^{-1/2}.   
\end{equation*}
Since $e^{-tA_s}$ is obtained through the functional calculus of the
positive self-adjoint operator $A_s$, it commutes with $A_s$.
Consequently, $B^*B=A_s^{-1}e^{-2tA_s}.$
If $\{(\lambda_j,e_j)\}_{j=1}^\infty$ denotes the spectral
decomposition of $A_s$, then
\begin{gather*}
    B^*Be_j=\lambda_j^{-1}e^{-2t\lambda_j}e_j,
\end{gather*}
so that the abstract source condition and the statistical error
estimates of Section~3 admit an explicit spectral characterization.
The following theorem summarizes the resulting reconstruction theory for
the backward fractional heat equation.
\begin{theorem}[Fractional heat semigroup observation]\label{thm-heat}
Let $A_s=I+(-\Delta)^s$ and let $t>0$. Define
$K=e^{-tA_s}.$
Then the estimator
\begin{gather*}
 u_\alpha^\delta=\operatorname*{arg\,min}_{u\in H_0^s(\Omega)}\left\{\frac12\|e^{-tA_s}u-y^\delta\|_{L^2(\Omega)}^2 +\frac{\alpha}{2}\|u\|_{H_0^s(\Omega)}^2 \right\}   
\end{gather*}
is well-defined and unique. Moreover, if
\begin{equation*}
A_s^{1/2}u^\dagger=(B^*B)^\beta w,\,B=e^{-tA_s}A_s^{-1/2},
\end{equation*}
with $w\in L^2(\Omega)$ and $0<\beta\le1$, then
\begin{equation*}
\mathbb E \|u_\alpha^\delta-u^\dagger\|_{H_0^s(\Omega)}^2
\le C_1\alpha^{2\beta}+ C_2\frac{\delta^2}{\alpha}.
\end{equation*}
\end{theorem}
\begin{proof}
Since $A_s\ge I$ is positive and self-adjoint, $e^{-tA_s}$ is bounded
on $L^2(\Omega)$. Hence
\begin{gather*}
    K=e^{-tA_s}:H_0^s(\Omega)\longrightarrow L^2(\Omega)
\end{gather*}
is bounded, and Theorem~\ref{thm:existence} gives existence and uniqueness of $u_\alpha^\delta$.
Set $B=e^{-tA_s}A_s^{-1/2}.$
Both factors are bounded Borel functions of the same self-adjoint operator
$A_s$; hence they commute and are self-adjoint. Therefore
\begin{equation*}
B^*B=A_s^{-1/2}e^{-2tA_s}A_s^{-1/2} =A_s^{-1}e^{-2tA_s}.
\end{equation*}
The assumed condition
$A_s^{1/2}u^\dagger=(B^*B)^\beta w$
is exactly the source condition in Theorem~\ref{thm:estimate_explicit}.
Applying that theorem yields
\begin{equation*}
\mathbb E\|u_\alpha^\delta-u^\dagger\|_{H_0^s(\Omega)}^2 \le C_1\alpha^{2\beta}+C_2\frac{\delta^2}{\alpha}.
\end{equation*}
\end{proof}

\begin{remark}
 The source condition appearing in Theorem~\ref{thm:estimate_explicit} is formulated in the
Hilbert-scale variable
$z^\dagger=A_s^{1/2}u^\dagger.$
For the backward fractional heat equation,
\[
B=e^{-tA_s}A_s^{-1/2},
\]
so that
$B^*B=A_s^{-1}e^{-2tA_s}.$
Consequently,
$z^\dagger=(B^*B)^\beta w$
is equivalent to
\[
u^\dagger
=
A_s^{-(\beta+\frac12)}
e^{-2\beta tA_s}
w.
\]
Thus the abstract Hilbert-scale source condition translates into a
fractional regularity assumption on the unknown solution involving the
fractional elliptic operator $A_s$ together with the smoothing effect
of the heat semigroup.
\end{remark}

\subsection{Spectral Interpretation of the Backward Fractional Heat Inverse Problem}
\label{sec:spectral-interpretation}
The backward fractional heat equation is distinguished by a common
spectral structure: the forward operator $K=e^{-tA_s}$, the
regularization penalty
$\|\cdot\|_{H_0^s(\Omega)}=\|A_s^{1/2}\cdot\|_{L^2(\Omega)}$, and the
Hilbert-scale operator $B^*B$ are all functions of the same operator
$A_s$, hence simultaneously diagonalized by its eigenbasis. This
subsection exploits that structure to make the reconstruction theory
fully explicit: we identify the smoothness forced by the source
condition, the decay rate of the singular values of $B$, and the
effective number of frequencies recoverable at a given noise level.  

We recall that
$A_s=I+(-\Delta)^s$; if $(-\Delta)^s e_j=\mu_j e_j,$
then
\[
A_s e_j=\lambda_j e_j,
\qquad
\lambda_j=1+\mu_j.
\]
Thus $A_s$ and $(-\Delta)^s$ have the same eigenfunctions, and their
eigenvalues differ only by the additive shift \(1\).

We first record the spectral decomposition of $A_s$ already used
informally in the proof of Theorem~\ref{thm-heat}.

\begin{lemma}[Discrete spectrum of $A_s$]
\label{lem:spectral-decomp}
Let $\Omega\subset\mathbb R^N$ be a bounded Lipschitz domain and
$s\in(0,1)$. The embedding $H_0^s(\Omega)\hookrightarrow
L^2(\Omega)$ is compact, and consequently $A_s=I+(-\Delta)^s$ has
compact resolvent on $L^2(\Omega)$. In particular, there exists an
orthonormal basis $\{e_j\}_{j=1}^\infty$ of $L^2(\Omega)$ and a
nondecreasing sequence $1\le\lambda_1\le\lambda_2\le\cdots\to\infty$
such that
$A_se_j=\lambda_je_j,\, j=1,2,\dots.$
\end{lemma}

\begin{proof}
The compactness of \(H_0^s(\Omega)\hookrightarrow L^2(\Omega)\) for a bounded Lipschitz domain and \(s\in(0,1)\) is classical; see \cite[Thm.~7.1]{DiNezza12}. Since \(A_s\ge I\) is the positive self-adjoint operator associated with the closed coercive form $a_s$ on the compactly embedded form domain $D(A_s^{1/2})=H_0^s(\Omega)$, the resolvent
$(A_s+I)^{-1}:L^2(\Omega)\to H_0^s(\Omega)\hookrightarrow L^2(\Omega)$ is compact, being the composition of a bounded map with a compact embedding. The spectral theorem for self-adjoint operators with compact resolvent then yields a discrete spectrum $\{\lambda_j\}\subset[1,\infty), \lambda_j\to\infty$, and a
corresponding orthonormal basis of eigenvectors.
\end{proof}

Throughout this subsection, $\{\lambda_j,e_j\}$ denotes the spectral data of Lemma~\ref{lem:spectral-decomp}, and we write $\mu_j:=\lambda_j-1\ge0$ for the eigenvalues of $(-\Delta)^s$ itself. We expand 
$u^\dagger=\sum_ju_je_j, \, w=\sum_jw_je_j$, and recall from the proof of Theorem~\ref{thm-heat} that $Be_j=\sigma_je_j$ with
\begin{gather*}
\sigma_j:=\lambda_j^{-1/2}e^{-t\lambda_j}, \,
\sigma_j^2=\lambda_j^{-1}e^{-2t\lambda_j}.
\end{gather*}

\subsubsection{Interpretation of the source condition}
\label{sec:spectral-source}
We first make the abstract source condition of Theorem~\ref{thm-heat}
explicit at the level of Fourier--$A_s$ coefficients, and identify the
regularity it forces on \(u^\dagger\).

\begin{proposition}[Mode-wise form of the source condition]
\label{prop:source-mode}
Let $\beta\in(0,1]$ and $w\in L^2(\Omega)$. The source condition
$A_s^{1/2}u^\dagger=(B^*B)^\beta w$ of Theorem~\ref{thm-heat} holds
if and only if
\begin{equation}
\label{eq:mode-source}
u_j=\lambda_j^{-\left(\beta+\frac12\right)}e^{-2\beta t\lambda_j}w_j, \,
j=1,2,\dots,
\end{equation}
where $\{w_j\}_{j\ge1}\in\ell^2$ with $\sum_jw_j^2=\|w\|_{L^2(\Omega)}^2$.
\end{proposition}

\begin{proof}
Since $B^*Be_j=\sigma_j^2e_j=\lambda_j^{-1}e^{-2t\lambda_j}e_j$, the spectral theorem gives $(B^*B)^\beta e_j=\lambda_j^{-\beta}e^{-2\beta t\lambda_j}e_j$, so the $j$-th coefficient of $(B^*B)^\beta w$ is $\lambda_j^{-\beta}e^{-2\beta t\lambda_j}w_j$. Equating this to the $j$-th coefficient $\lambda_j^{1/2}u_j$ of $A_s^{1/2}u^\dagger$ and solving for $u_j$ gives \eqref{eq:mode-source}.
\end{proof}

The exponential factor in \eqref{eq:mode-source} dominates every
polynomial factor as $\lambda_j\to\infty$, so the source condition
is far more restrictive than any fixed-order fractional Sobolev
condition on $u^\dagger$: it forces $u^\dagger$ to belong to the
space $C^\infty(A_s):=\bigcap_{q\ge0}D(A_s^q)$ of $A_s$-smooth
vectors, with quantitatively controlled, factorially growing
$A_s$-Sobolev norms.

\begin{proposition}[Gevrey-type regularity forced by the source condition]
\label{prop:gevrey}
Under the hypotheses of Proposition~\ref{prop:source-mode}, assume in
addition that \(\|w\|_{L^2(\Omega)}\le\rho_0\). Then
\[
u^\dagger\in D(A_s^q)
\qquad\text{for every }q\ge0.
\]
Moreover, for every \(q>\beta+\tfrac12\),
\begin{equation}\label{eq:gevrey-bound}
\|A_s^q u^\dagger\|_{L^2(\Omega)}
\le
\rho_0\max\left\{
e^{-2\beta t},
\left(
\frac{2q-2\beta-1}{4e\beta t}
\right)^{q-\beta-\frac12}
\right\}.
\end{equation}
In particular, \eqref{eq:gevrey-bound} gives a factorial-type upper bound
for \(\|A_s^q u^\dagger\|_{L^2(\Omega)}\) as \(q\to\infty\), characteristic
of analytic, or Gevrey-type, vectors associated with \(A_s\).
\end{proposition}

\begin{proof}
By \eqref{eq:mode-source},
\[
\|A_s^q u^\dagger\|_{L^2(\Omega)}^2
=
\sum_{j=1}^\infty \lambda_j^{2q}u_j^2
=
\sum_{j=1}^\infty
\lambda_j^r e^{-4\beta t\lambda_j}w_j^2,
\qquad r:=2q-2\beta-1.
\]
Since \(\lambda_j\ge1\),
\[
\|A_s^q u^\dagger\|_{L^2(\Omega)}^2
\le
\left(\sup_{\lambda\ge1}\lambda^r e^{-4\beta t\lambda}\right)
\|w\|_{L^2(\Omega)}^2.
\]
For every \(r\in\mathbb R\), the function
\(f(\lambda)=\lambda^r e^{-4\beta t\lambda}\) is continuous on
\([1,\infty)\) and tends to zero as \(\lambda\to\infty\). Thus the
supremum is finite, proving \(u^\dagger\in D(A_s^q)\) for all \(q\ge0\).

Now let \(q>\beta+\tfrac12\), so that \(r>0\). Since
\[
f'(\lambda)=\lambda^{r-1}e^{-4\beta t\lambda}
\bigl(r-4\beta t\lambda\bigr),
\]
the unique critical point is \(\lambda^*=r/(4\beta t)\). If
\(\lambda^*\ge1\), the maximum on \([1,\infty)\) is attained at
\(\lambda^*\); if \(\lambda^*<1\), then \(f\) is decreasing on
\([1,\infty)\) and its maximum is attained at \(1\). Consequently,
\[
\sup_{\lambda\ge1}\lambda^r e^{-4\beta t\lambda}
\le
\max\left\{
e^{-4\beta t},
\left(\frac{r}{4e\beta t}\right)^r
\right\}.
\]
Using \(\|w\|_{L^2(\Omega)}\le\rho_0\), taking square roots, and
substituting \(r=2q-2\beta-1\) gives \eqref{eq:gevrey-bound}.

\end{proof}
Thus \(u^\dagger\) is forced to be smooth on the \(A_s\)-scale to
every finite order, with norms growing at the factorial rate
\eqref{eq:gevrey-bound} regardless of \(\beta\) -- the sense in which
the backward fractional heat problem is \emph{severely}, rather than
mildly, ill posed.

\subsubsection{Singular value asymptotics}
\label{sec:spectral-sv}

The mode-wise picture of Section~\ref{sec:spectral-source} becomes
quantitative once the growth of \(\lambda_j\) in \(j\) is known. This
is furnished by the classical Weyl-type eigenvalue asymptotics for
the Dirichlet fractional Laplacian, due to Blumenthal and Getoor.

\begin{proposition}[Weyl asymptotics for $A_s$; \cite{BlumenthalGetoor1959}, see also {\cite[Thm.~3.1]{FrankReview2016}}]
\label{prop:weyl}
Let \(\Omega\subset\mathbb R^N\) be open and bounded, and let
\(\{\mu_j\}_{j\ge1}\) be the eigenvalues corresponding to the restricted Dirichlet fractional
Laplacian (that is, the integral fractional Laplacian with homogeneous
exterior Dirichlet condition)
\((-\Delta)^s\) on \(\Omega\), arranged
nondecreasingly. Then
\begin{equation}
\label{eq:weyl}
\mu_j
\sim
\kappa_{N,s}\,|\Omega|^{-2s/N}\,j^{2s/N}
\qquad(j\to\infty),
\qquad
\kappa_{N,s}:=(2\pi)^{2s}\omega_N^{-2s/N},
\end{equation}
where \(\omega_N=|\{\xi\in\mathbb R^N:|\xi|<1\}|\). Equivalently, the
eigenvalue counting function \(N(\Lambda):=\#\{j:\mu_j\le\Lambda\}\)
satisfies
\begin{equation}
\label{eq:weyl-counting}
N(\Lambda)
\sim
(2\pi)^{-N}\omega_N\,|\Omega|\,\Lambda^{N/(2s)}
\qquad(\Lambda\to\infty).
\end{equation}
Consequently \(\lambda_j=1+\mu_j\) satisfies
\(\lambda_j\sim\kappa_{N,s}|\Omega|^{-2s/N}j^{2s/N}\) as well, and the
same counting-function asymptotics \eqref{eq:weyl-counting} hold with
\(\mu_j\) replaced by \(\lambda_j\).
\end{proposition}

\begin{proof}
This is the classical result of Blumenthal and Getoor
\cite{BlumenthalGetoor1959} for the Dirichlet heat kernel of a
rotationally symmetric \(2s\)-stable process killed upon exiting
\(\Omega\), whose generator is exactly the restricted fractional
Laplacian \((-\Delta)^s\) on \(H_0^s(\Omega)\) used in this paper;
see \cite[Thm.~3.1]{FrankReview2016} for a modern statement in this
normalization. Since \(\lambda_j=1+\mu_j\) and \(\mu_j\to\infty\), the
additive shift by \(1\) is asymptotically negligible, giving the
stated asymptotics for \(\lambda_j\) and for the corresponding
counting function.
\end{proof}

\begin{corollary}[Singular value decay of $B$]
\label{cor:sv-decay}
The singular values \(\sigma_j=\lambda_j^{-1/2}e^{-t\lambda_j}\) of
\(B=e^{-tA_s}A_s^{-1/2}\) satisfy
\[
\log\frac1{\sigma_j}
=
t\lambda_j+\tfrac12\log\lambda_j
\sim
t\,\kappa_{N,s}\,|\Omega|^{-2s/N}\,j^{2s/N}
\qquad(j\to\infty).
\]
In particular \(\sigma_j\to0\) faster than any negative power of
\(j\), so \(B\) is compact and belongs to every Schatten class
\(\mathcal S_p(L^2(\Omega))\), \(p>0\); the inverse problem is
\emph{exponentially}, rather than polynomially, ill posed.
\end{corollary}

\begin{proof}
Immediate from Proposition~\ref{prop:weyl} and the definition of
\(\sigma_j\), since \(t\lambda_j\) dominates
\(\tfrac12\log\lambda_j\) as \(j\to\infty\). Because
\(\log(1/\sigma_j)\) grows like \(j^{2s/N}\), for every \(p>0\) we
have \(\sum_j\sigma_j^p<\infty\), which is the definition of
membership in \(\mathcal S_p(L^2(\Omega))\).
\end{proof}

\subsubsection{Dependence of ill-posedness on $t$ and $s$}
\label{sec:spectral-severity}

Corollary~\ref{cor:sv-decay} makes the dependence of the severity of
ill-posedness on the observation time \(t\) and the fractional order
\(s\) transparent.

\begin{remark}[Dependence on $t$]
For fixed \(j\), \(\sigma_j(t)=\lambda_j^{-1/2}e^{-t\lambda_j}\) is
strictly decreasing in \(t\), with \(\sigma_j(t)\to\lambda_j^{-1/2}\)
as \(t\to0^+\); in this limit \(K\to I\) and \(B\to A_s^{-1/2}\), so
the problem degenerates into the \emph{mildly} (polynomially)
ill-posed inversion of \(A_s^{1/2}\), with \(\sigma_j\sim j^{-s/N}\)
by Proposition~\ref{prop:weyl}. The severe ill-posedness of the
backward heat problem is thus entirely a consequence of the positive
observation time \(t>0\), with \(\log(1/\sigma_j)\) growing linearly
in \(t\).
\end{remark}

\begin{remark}[Dependence on $s$]
By \eqref{eq:weyl-counting}, the eigenvalue-counting exponent
$N/2s$ is \emph{decreasing} in $s$: larger $s$ concentrates
the useful low-frequency information in fewer modes, so increasing
$s$ makes the backward problem more severely ill posed. This is
consistent with $(-\Delta)^s$ becoming more strongly (locally)
smoothing as $s\to1^-$, where the classical, most severely
ill-posed backward heat problem is recovered, matching the
Bourgain--Brezis--Mironescu limit of Section~\ref{sec:Tikhonov-model}.
Thus $s$ acts as a continuous dial on the severity of
ill-posedness, interpolating between the mildly ill-posed limit
$t\to0^+$ above and the maximally severe classical limit $s\to1^-$.
\end{remark}

\subsubsection{Reconstruction error interpreted through the spectrum}
\label{sec:spectral-error}

We now revisit the mean-square estimate of Theorem~\ref{thm-heat} at
the level of individual Fourier--$A_s$ modes, which makes precise how
many frequencies Tikhonov regularization actually recovers at a given
noise level.

\begin{proposition}[Mode-wise filtering]
\label{prop:mode-filter}
The minimizer \(z_\alpha^\delta=A_s^{1/2}u_\alpha^\delta\) of
Theorem~\ref{thm-heat} satisfies, coefficientwise,
\begin{equation}
\label{eq:filter}
(z_\alpha^\delta)_j
=\frac{\sigma_j^2}{\sigma_j^2+\alpha}\,z_j^\dagger+\frac{\sigma_j}{\sigma_j^2+\alpha}\,(\eta^\delta)_j,
\,
j=1,2,\dots,
\end{equation}
where \(z_j^\dagger=\lambda_j^{1/2}u_j\) and
\((\eta^\delta)_j=\langle\eta^\delta,e_j\rangle\). Consequently
\begin{equation}
\label{eq:mode-mse}
\mathbb E\bigl|(z_\alpha^\delta)_j-z_j^\dagger\bigr|^2
\le
2\left(\frac{\alpha}{\sigma_j^2+\alpha}\right)^2|z_j^\dagger|^2
+
2\left(\frac{\sigma_j}{\sigma_j^2+\alpha}\right)^2
\mathbb E\bigl|(\eta^\delta)_j\bigr|^2.
\end{equation}
\end{proposition}

\begin{proof}
Since \(T=B^*B\) is diagonal in \(\{e_j\}\) with eigenvalues
\(\sigma_j^2\), the normal equation
\((T+\alpha I)z_\alpha^\delta=Tz^\dagger+B^*\eta^\delta\) from the
proof of Theorem~\ref{thm-heat} decouples mode by mode into
\((\sigma_j^2+\alpha)(z_\alpha^\delta)_j=\sigma_j^2z_j^\dagger+\sigma_j(\eta^\delta)_j\),
using \(\langle B^*\eta^\delta,e_j\rangle=\langle\eta^\delta,Be_j\rangle=\sigma_j(\eta^\delta)_j\).
Solving for \((z_\alpha^\delta)_j\) gives \eqref{eq:filter}; writing
\((z_\alpha^\delta)_j-z_j^\dagger\) as the sum of the two terms in
\eqref{eq:filter} minus \(z_j^\dagger\), namely
\(-\frac{\alpha}{\sigma_j^2+\alpha}z_j^\dagger+\frac{\sigma_j}{\sigma_j^2+\alpha}(\eta^\delta)_j\),
and applying \(|a+b|^2\le2|a|^2+2|b|^2\) followed by expectation
gives \eqref{eq:mode-mse}.
\end{proof}

The filter factor \(\alpha/(\sigma_j^2+\alpha)\) in
\eqref{eq:mode-mse} is close to \(0\) for well-resolved modes
(\(\sigma_j^2\gg\alpha\)) and close to \(1\) for suppressed modes
(\(\sigma_j^2\ll\alpha\)); since \(\sigma_j^2=\lambda_j^{-1}e^{-2t\lambda_j}\)
decays exponentially by Corollary~\ref{cor:sv-decay}, this transition
is sharp, and only finitely many low modes are ever trusted for a
given \(\alpha\). The transition threshold can be located exactly.

\begin{proposition}[Effective bandwidth]
\label{prop:bandwidth}
For \(0<\alpha\le e^{-2t}\), the equation \(\sigma_j^2=\alpha\), i.e.\
\(\lambda^{-1}e^{-2t\lambda}=\alpha\), has a unique solution
\(\lambda^*(\alpha)\ge1\) among \(\lambda\ge1\), given explicitly by
\begin{equation}
\label{eq:lambertw}
\lambda^*(\alpha)=\frac1{2t}\,W\!\left(\frac{2t}{\alpha}\right),
\end{equation}
where \(W\) denotes the principal branch of the Lambert $W$ function.
As \(\alpha\to0^+\),
\[
\lambda^*(\alpha)\sim\frac1{2t}\log\frac1\alpha.
\]
Consequently, the number of modes with \(\lambda_j\le\lambda^*(\alpha)\)
-- the ``effective bandwidth'' of the regularized reconstruction --
satisfies, by Proposition~\ref{prop:weyl},
\begin{equation}
\label{eq:effective-modes}
j^*(\alpha):=N(\lambda^*(\alpha))
\sim
(2\pi)^{-N}\omega_N|\Omega|
\left(\frac1{2t}\log\frac1\alpha\right)^{N/(2s)}
\qquad(\alpha\to0^+).
\end{equation}
In particular, under the optimal a priori choice
\(\alpha=\delta^{2/(2\beta+1)}\) of Theorem~\ref{thm-heat}, the
effective bandwidth grows only \emph{poly-logarithmically} in the
inverse noise level:
\begin{equation}
\label{eq:effective-modes-delta}
j^*(\delta)
\sim
(2\pi)^{-N}\omega_N|\Omega|
\left(\frac1{t(2\beta+1)}\log\frac1\delta\right)^{N/(2s)}
\qquad(\delta\to0^+).
\end{equation}
\end{proposition}

\begin{proof}
The function \(f(\lambda)=\lambda^{-1}e^{-2t\lambda}\) is strictly
decreasing on \((0,\infty)\), since
\(f'(\lambda)=-e^{-2t\lambda}(\lambda^{-2}+2t\lambda^{-1})<0\); as
\(f(1)=e^{-2t}\ge\alpha\) and \(f(\lambda)\to0\) as
\(\lambda\to\infty\), there is a unique \(\lambda^*(\alpha)\ge1\)
with \(f(\lambda^*(\alpha))=\alpha\). Writing \(y=2t\lambda\), the
equation \(f(\lambda)=\alpha\) becomes \(y+\log y=\log(2t/\alpha)\),
i.e.\ \(ye^y=2t/\alpha\), whose unique positive solution is
\(y=W(2t/\alpha)\) by definition of the Lambert \(W\) function; this
gives \eqref{eq:lambertw}. The asymptotics
\(W(z)\sim\log z-\log\log z\sim\log z\) as \(z\to\infty\) give
\(\lambda^*(\alpha)\sim\frac1{2t}\log(2t/\alpha)\sim\frac1{2t}\log(1/\alpha)\)
as \(\alpha\to0^+\). Substituting into \eqref{eq:weyl-counting} gives
\eqref{eq:effective-modes}, and substituting
\(\log(1/\alpha)=\frac2{2\beta+1}\log(1/\delta)\) (from
\(\alpha=\delta^{2/(2\beta+1)}\)) gives \eqref{eq:effective-modes-delta}.
\end{proof}

This exhibits the same tension already visible in
Proposition~\ref{prop:gevrey}: the number of individually
trustworthy frequencies grows only poly-logarithmically in
\(1/\delta\), yet the \emph{aggregate} reconstruction error still
decays at the polynomial rate \(\delta^{4\beta/(2\beta+1)}\) of
Theorem~\ref{thm-heat}. The two facts are compatible because the
source condition forces the true coefficients \(z_j^\dagger\) to
decay exponentially in \(\lambda_j\): the bias contributed by every
mode beyond the effective bandwidth \(j^*(\alpha)\) is already
negligible before regularization is applied, so the truncation
implicit in \eqref{eq:filter}--\eqref{eq:mode-mse} discards almost no
signal, only noise.

\subsubsection{Why the fractional penalty is naturally matched to the forward model}
\label{sec:spectral-matching}

\begin{remark}[Matched versus mismatched regularization]
The explicit results above rely on \(K=e^{-tA_s}\), the
regularization operator \(A_s\), and \(T=B^*B=A_s^{-1}e^{-2tA_s}\)
all being functions of the same operator \(A_s\), hence simultaneously
diagonalizable in the eigenbasis \(\{e_j\}\) of
Lemma~\ref{lem:spectral-decomp}: this is what lets the source
condition be written mode-wise as in \eqref{eq:mode-source} and the
resolution threshold computed explicitly as in \eqref{eq:lambertw}.

If the penalty were instead generated by an operator not spectrally
matched to the forward map -- for instance the classical Dirichlet
energy against a forward model governed by \((-\Delta)^s\), or more
generally any observation operator not commuting with \(A_s\) -- the
source condition could no longer be reduced to \eqref{eq:mode-source},
and the general theory of exponentially ill-posed problems predicts
only \emph{logarithmic} convergence rates; see Mair \cite{Mair1994}
and Hohage \cite{Hohage2000}. Matching the regularization operator to
the forward operator, automatic here since both are functions of
\(A_s\), is thus what upgrades the generic logarithmic rate to the
polynomial rate \(\delta^{4\beta/(2\beta+1)}\) of Theorem~\ref{thm-heat}.
\end{remark}

Together, the two applications illustrate complementary strengths of
the framework: partial observation shows its applicability to a
generic bounded observation operator, while the backward fractional
heat equation shows how a matched spectral structure yields fully
explicit source conditions, constants, and convergence rates. We now
turn to the asymptotic behavior of the framework as the fractional
order approaches the classical limit \(s\to1^-\).

\section{The Classical Limit as $s\to1^{-}$}\label{sec:Tikhonov-model}

Throughout this section, we assume $K\in\mathcal L(L^2(\Omega),L^2(\Omega)).$ We choose the normalization constant $C_{N,s}$ so that, for every $u\in H_0^1(\Omega)$,
\begin{gather*}
  \lim_{s\to1^-}
\frac{C_{N,s}}{2}
\iint_{\mathbb R^{2N}}
\frac{|u(x)-u(y)|^2}{|x-y|^{N+2s}}\,dx\,dy
=\int_\Omega |\nabla u|^2\,dx.  
\end{gather*}

In particular, $C_{N,s}\asymp1-s$ as $s\to1^-$. We write
\begin{equation}\label{eq:E}
\mathcal{E}_s(u,v)=\frac{C_{N,s}}{2}
\iint_{\mathbb R^{2N}}
\frac{(u(x)-u(y))(v(x)-v(y))}{|x-y|^{N+2s}}\,dx\,dy.
\end{equation}

\subsection{BBM Limit for the Functional}
The previous sections established the fractional regularization framework and demonstrated its applicability to representative inverse problems. We now investigate its asymptotic behaviour as the fractional order approaches the local limit $s\to1^{-}$. This analysis provides a rigorous justification that the proposed nonlocal model is a genuine extension of the classical Tikhonov regularization. The first theorem establishes the convergence of the fractional regularization functional through the Bourgain--Brezis--Mironescu limit.
\begin{theorem}[Pointwise convergence of the functional]
\label{thm:bbm-functional}
Let $K:H_0^1(\Omega)\to L^2(\Omega)$ be bounded and linear, let
$y^\delta\in L^2(\Omega)$, and let $\alpha>0$. For $s\in(0,1)$,
define
\begin{equation}\label{eq:J-al-s}
\mathcal{J}_{\alpha,s}(u)
=
\frac12\|Ku-y^\delta\|_{L^2(\Omega)}^2
+
\frac{\alpha}{2}
\left(
\|u\|_{L^2(\Omega)}^2+\mathcal E_s(u,u)
\right).
\end{equation}
Then, for every fixed $u\in H_0^1(\Omega)$,
\[
\lim_{s\to1^-}\mathcal{J}_{\alpha,s}(u)
=
\mathcal{J}_{\alpha,1}(u),
\]
where
\begin{equation}\label{eq:J-al-1}
\mathcal{J}_{\alpha,1}(u)
=
\frac12\|Ku-y^\delta\|_{L^2(\Omega)}^2
+
\frac{\alpha}{2}
\left(
\|u\|_{L^2(\Omega)}^2+
\|\nabla u\|_{L^2(\Omega)}^2
\right).
\end{equation}
\end{theorem}

\begin{proof}
Fix $u\in H_0^1(\Omega)$. The terms
$\|Ku-y^\delta\|_{L^2(\Omega)}^2
\,\text{and}\, \|u\|_{L^2(\Omega)}^2$
do not depend on $s$. Hence the only term that must be analyzed is the
fractional energy $\mathcal E_s(u,u)$. By the
Bourgain--Brezis--Mironescu formula,
\[
\lim_{s\to1^-}
\mathcal E_s(u,u)
=
\int_\Omega |\nabla u|^2\,dx.
\]
Therefore
\[
\begin{aligned}
\lim_{s\to1^-} \mathcal{J}_{\alpha,s}(u)
&=
\frac12\|Ku-y^\delta\|_{L^2(\Omega)}^2+
\frac{\alpha}{2}\|u\|_{L^2(\Omega)}^2+
\frac{\alpha}{2} \int_\Omega |\nabla u|^2\,dx\\
&=
\mathcal{J}_{\alpha,1}(u).
\end{aligned}
\]
This proves the theorem.
\end{proof}

\subsection{BBM Limit for the Euler--Lagrange Equation}
The convergence of the regularization functional established in the previous subsection does not, by itself, imply the convergence of the corresponding optimality conditions. We therefore investigate the asymptotic behaviour of the Euler--Lagrange equation associated with the fractional Tikhonov functional. The following theorem shows that, as $s\to1^{-}$, the fractional optimality system converges to the classical Euler--Lagrange equation, thereby completing the rigorous asymptotic connection between the fractional and classical regularization frameworks. We refer to Bourgain, Brezis and Mironescu \cite{BourgainBrezisMironescu2001}, Ponce \cite{Ponce2004}, and D\'avila~\cite{Davila02} for the Bourgain--Brezis--Mironescu formula and its subsequent developments.
\begin{theorem}[Convergence of the Euler--Lagrange bilinear forms]
\label{thm:bbm-euler}
Let $u,v\in H_0^1(\Omega)$. Then
\[
\lim_{s\to1^-}\mathcal E_s(u,v)
=
\int_\Omega \nabla u\cdot\nabla v\,dx.
\]
Consequently, the weak fractional Euler--Lagrange form
\[
(Ku-y^\delta,Kv)_{L^2(\Omega)}
+
\alpha (u,v)_{L^2(\Omega)}
+
\alpha\mathcal E_s(u,v)
=0
\]
converges, as $s\to1^-$, to the classical weak Euler--Lagrange form
\[
(Ku-y^\delta,Kv)_{L^2(\Omega)}
+
\alpha (u,v)_{L^2(\Omega)}
+
\alpha\int_\Omega \nabla u\cdot\nabla v\,dx
=0.
\]
Equivalently, the fractional operator equation
\[
K^*(Ku-y^\delta)+\alpha u+\alpha(-\Delta)^s u=0
\]
converges formally in weak form to
\[
K^*(Ku-y^\delta)+\alpha u-\alpha\Delta u=0.
\]
\end{theorem}

\begin{proof}
Let $u,v\in H_0^1(\Omega)$. By the polarization identity,
\[
\mathcal E_s(u,v)
=
\frac14
\left[
\mathcal E_s(u+v,u+v)
-
\mathcal E_s(u-v,u-v)
\right].
\]
Since $u+v\in H_0^1(\Omega)$ and $u-v\in H_0^1(\Omega)$, the
Bourgain--Brezis--Mironescu formula gives
\[
\lim_{s\to1^-}\mathcal E_s(u+v,u+v)
=
\int_\Omega |\nabla(u+v)|^2\,dx
\]
and
\[
\lim_{s\to1^-}\mathcal E_s(u-v,u-v)
=
\int_\Omega |\nabla(u-v)|^2\,dx.
\]
Therefore,
\[
\begin{aligned}
\lim_{s\to1^-}\mathcal E_s(u,v)
&=
\frac14
\left[
\int_\Omega |\nabla(u+v)|^2\,dx
-
\int_\Omega |\nabla(u-v)|^2\,dx
\right] \\
&=
\int_\Omega \nabla u\cdot\nabla v\,dx.
\end{aligned}
\]
Now observe that the remaining two terms in the weak Euler--Lagrange
identity,
\[
(Ku-y^\delta,Kv)_{L^2(\Omega)}
\quad\text{and}\quad
(u,v)_{L^2(\Omega)},
\]
do not depend on \(s\). Passing to the limit in
\[
(Ku-y^\delta,Kv)_{L^2(\Omega)}
+
\alpha (u,v)_{L^2(\Omega)}
+
\alpha\mathcal E_s(u,v)
=0
\]
therefore yields
\[
(Ku-y^\delta,Kv)_{L^2(\Omega)}
+
\alpha (u,v)_{L^2(\Omega)}
+
\alpha\int_\Omega \nabla u\cdot\nabla v\,dx
=0.
\]
This is precisely the weak Euler--Lagrange equation corresponding to
the classical Tikhonov functional \(J_{\alpha,1}\).
\end{proof}
\begin{remark}
The convergence result established in this section shows that the proposed
fractional Tikhonov model is a genuine extension of the classical
least-squares regularization framework. Indeed, after the
Bourgain--Brezis--Mironescu normalization, the fractional Dirichlet energy
converges to the classical Dirichlet energy as $s\to1^{-}$.
Consequently, the fractional variational problem continuously recovers the
classical elliptic Tikhonov regularization model.

This limiting behaviour demonstrates that the present theory unifies both
local and nonlocal regularization within a single variational framework,
with the fractional order $s$ acting as a parameter that interpolates
between the two regimes.
\end{remark}
\subsection{Gamma-Convergence and Convergence of Minimizers}
Throughout this subsection, we assume that
\[
K:L^2(\Omega)\longrightarrow L^2(\Omega)
\]
is a bounded linear operator. All $\Gamma$-convergence statements are
understood with respect to the strong topology of $L^2(\Omega)$.
We now strengthen the pointwise Bourgain--Brezis--Mironescu limit by
showing that the full fractional Tikhonov functionals converge, in the
sense of $\Gamma$-convergence, to the classical Tikhonov functional.
This provides convergence not only of the energies, but also of the
corresponding minimizers.

For \(s\in(0,1)\), we define
\[
\mathcal{E}_s(u,u)
=
\frac{C_{N,s}}{2}
\iint_{\mathbb R^{2N}}
\frac{|u(x)-u(y)|^2}{|x-y|^{N+2s}}
\,dx\,dy.
\]
We regard the fractional Tikhonov functional as a functional on
$L^2(\Omega)$ by setting
\[
\mathcal J_{\alpha,s}(u)
=
\begin{cases}
\displaystyle
\frac12\|Ku-y^\delta\|_{L^2(\Omega)}^2
+
\frac{\alpha}{2}
\left(
\|u\|_{L^2(\Omega)}^2+\mathcal{E}_s(u,u)
\right),
& u\in H_0^s(\Omega),\\[2ex]
+\infty,
& u\notin H_0^s(\Omega).
\end{cases}
\]
The limiting classical functional is
\[
\mathcal J_{\alpha,1}(u)
=
\begin{cases}
\displaystyle
\frac12\|Ku-y^\delta\|_{L^2(\Omega)}^2
+
\frac{\alpha}{2}
\left(
\|u\|_{L^2(\Omega)}^2+\|\nabla u\|_{L^2(\Omega)}^2
\right),
& u\in H_0^1(\Omega),\\[2ex]
+\infty,
& u\notin H_0^1(\Omega).
\end{cases}
\]

The first step is to record the coercivity of the fractional functionals.

\begin{lemma}[Coercivity]\label{lem:coercive}
Let $\alpha>0$ and $y^\delta\in L^2(\Omega)$. Then, for every
$s\in(0,1)$ and every $u\in L^2(\Omega)$,
\[
\mathcal J_{\alpha,s}(u)
\ge
\frac{\alpha}{2}\|u\|_{L^2(\Omega)}^2.
\]
Consequently, if $\{\mathcal J_{\alpha,s}(u_s)\}$ is bounded, then
$\{u_s\}$ is bounded in $L^2(\Omega)$.
\end{lemma}

\begin{proof}
If $u\notin H_0^s(\Omega)$, then
$\mathcal J_{\alpha,s}(u)=+\infty$, and the inequality is trivial.
Let $u\in H_0^s(\Omega)$. 
Both the fidelity term and the fractional energy $\mathcal{E}_s(u,u)$ are
nonnegative. Therefore
$\mathcal J_{\alpha,s}(u) \ge \frac{\alpha}{2}\|u\|_{L^2(\Omega)}^2.$
Hence, if $\mathcal J_{\alpha,s}(u_s)\le C$, then
$\frac{\alpha}{2}\|u_s\|_{L^2(\Omega)}^2\le C,$
and so $\|u_s\|_{L^2(\Omega)}^2\le \frac{2C}{\alpha}.$ Thus $\{u_s\}$ is bounded in $L^2(\Omega)$.
\end{proof}
Coercivity alone gives boundedness of a minimizing sequence; compactness
requires the Bourgain--Brezis--Mironescu theorem.

\begin{lemma}[Compactness of bounded fractional energies]\label{lem:BBMcompact}
Let \(\{s_n\}\subset(0,1)\) satisfy \(s_n\to1^{-}\), and let
\(\{u_n\}\subset L^2(\Omega)\) satisfy
\[
\sup_n\mathcal J_{\alpha,s_n}(u_n)<\infty.
\]
Then \(\{u_n\}\) is relatively compact in \(L^2(\Omega)\). In
particular, there exists a subsequence, still denoted by
\(\{u_n\}\), and a function \(u\in H_0^1(\Omega)\) such that
\[
u_n\to u
\qquad\text{strongly in }L^2(\Omega).
\]
Moreover,
\[
\|\nabla u\|_{L^2(\Omega)}^2
\le
\liminf_{n\to\infty}
\mathcal{E}_{s_n}(u_n,u_n).
\]
\end{lemma}

\begin{proof}
By Lemma~\ref{lem:coercive},
$\sup_n\|u_n\|_{L^2(\Omega)}<\infty.$
Moreover, we have,
$\sup_n \mathcal{E}_{s_n}(u_n,u_n)<\infty,$
since
\[
\mathcal J_{\alpha,s_n}(u_n)
=
\frac12\|Ku_n-y^\delta\|_{L^2(\Omega)}^2
+
\frac{\alpha}{2}
\|u_n\|_{L^2(\Omega)}^2
+
\frac{\alpha}{2}
\mathcal{E}_{s_n}(u_n,u_n).
\]
The Bourgain--Brezis--Mironescu compactness theorem
(see Bourgain, Brezis and Mironescu~\cite{BourgainBrezisMironescu2001},
Ponce~\cite{Ponce2004}, and D\'avila~\cite{Davila02})
therefore implies that $\{u_n\}$ is relatively compact in $L^2(\Omega)$.
Hence, after extraction of a subsequence,
$u_n\to u \text{ strongly in }L^2(\Omega),$
for some $u\in H_0^1(\Omega)$. Furthermore, the same theorem yields
\[
\|\nabla u\|_{L^2(\Omega)}^2
\le
\liminf_{n\to\infty}
\mathcal{E}_{s_n}(u_n,u_n),
\]
which completes the proof.
\end{proof}

With compactness in hand, it remains to control the fidelity term
along a convergent sequence.

\begin{lemma}[Continuity of the fidelity term]\label{lem:fidelity-cont}
Let \(K:L^2(\Omega)\to L^2(\Omega)\) be bounded and linear. If
\(u_n\to u\) strongly in \(L^2(\Omega)\), then
\[
\|Ku_n-y^\delta\|_{L^2(\Omega)}^2
\to
\|Ku-y^\delta\|_{L^2(\Omega)}^2.
\]
\end{lemma}

\begin{proof}
Since \(K\) is bounded on \(L^2(\Omega)\),
\[
\|Ku_n-Ku\|_{L^2(\Omega)}
\le
\|K\|\,\|u_n-u\|_{L^2(\Omega)}
\to 0.
\]
Hence \(Ku_n-y^\delta\to Ku-y^\delta\) strongly in \(L^2(\Omega)\), and
therefore the squared norms converge.
\end{proof}

The three lemmas above combine directly to give the \(\Gamma\)-liminf
inequality.

\begin{proposition}[Gamma-liminf inequality]\label{prop:gamma-liminf}
Let \(s_n\to1^{-}\) and let \(u_n\to u\) strongly in \(L^2(\Omega)\).
Then
\[
\mathcal J_{\alpha,1}(u)
\le
\liminf_{n\to\infty}
\mathcal J_{\alpha,s_n}(u_n).
\]
\end{proposition}

\begin{proof}
If
\[
\liminf_{n\to\infty}\mathcal J_{\alpha,s_n}(u_n)=+\infty,
\]
there is nothing to prove. Hence assume, up to a subsequence, that
\[
\sup_n \mathcal J_{\alpha,s_n}(u_n)<\infty.
\]
By Lemma~\ref{lem:BBMcompact}, the limit \(u\) belongs to $H_0^1(\Omega)$, and
\[
\|\nabla u\|_{L^2(\Omega)}^2
\le
\liminf_{n\to\infty}\mathcal{E}_{s_n}(u_n,u_n).
\]
Moreover, by Lemma~\ref{lem:fidelity-cont},
\[
\|Ku_n-y^\delta\|_{L^2(\Omega)}^2
\to
\|Ku-y^\delta\|_{L^2(\Omega)}^2,
\]
and by the strong \(L^2\)-convergence,
$\|u_n\|_{L^2(\Omega)}^2 \to \|u\|_{L^2(\Omega)}^2.$
Combining these three facts gives
\[
\begin{aligned}
\liminf_{n\to\infty}\mathcal J_{\alpha,s_n}(u_n)
&\ge
\frac12\|Ku-y^\delta\|_{L^2(\Omega)}^2
+
\frac{\alpha}{2}\|u\|_{L^2(\Omega)}^2
+
\frac{\alpha}{2}\|\nabla u\|_{L^2(\Omega)}^2 =
\mathcal J_{\alpha,1}(u).
\end{aligned}
\]
This proves the $\Gamma$-liminf inequality.
\end{proof}

It remains to produce a recovery sequence, which gives the matching
$\Gamma$-limsup inequality.

\begin{proposition}[Gamma-limsup inequality]\label{prop:gamma-limsup}
For every $u\in L^2(\Omega)$, there exists a sequence
$\{u_s\}_{s\in(0,1)}\subset L^2(\Omega)$ such that
\[
u_s\to u
\quad\text{strongly in }L^2(\Omega)
\]
and
\[
\mathcal J_{\alpha,1}(u)
\ge
\limsup_{s\to1^-}
\mathcal J_{\alpha,s}(u_s).
\]
\end{proposition}

\begin{proof}
If $u\notin H_0^1(\Omega)$, then $\mathcal J_{\alpha,1}(u)=+\infty$. Hence assume $u\in H_0^1(\Omega)$. We choose the constant recovery sequence $u_s=u.$ Then clearly
\[
u_s\to u \, \text{strongly in }L^2(\Omega).
\]
Moreover, since $u\in H_0^1(\Omega)$, we have $u\in H_0^s(\Omega)$
for every $s\in(0,1)$. Hence $\mathcal J_{\alpha,s}(u_s)$ is finite.

By the Bourgain--Brezis--Mironescu formula,
\[
\mathcal{E}_s(u,u)\to \|\nabla u\|_{L^2(\Omega)}^2
\qquad\text{as }s\to1^-.
\]
The remaining terms in the functional do not depend on $s$. Therefore
\[
\begin{aligned}
\lim_{s\to1^-}\mathcal J_{\alpha,s}(u)
&=
\frac12\|Ku-y^\delta\|_{L^2(\Omega)}^2
+
\frac{\alpha}{2}\|u\|_{L^2(\Omega)}^2
+
\frac{\alpha}{2}\|\nabla u\|_{L^2(\Omega)}^2=
\mathcal J_{\alpha,1}(u).
\end{aligned}
\]
Thus
\[
\limsup_{s\to1^-}\mathcal J_{\alpha,s}(u_s)
\le
\mathcal J_{\alpha,1}(u),
\]
which proves the recovery sequence property.
\end{proof}

Together, the liminf and limsup inequalities give $\Gamma$-convergence.

\begin{theorem}[Gamma-convergence of the fractional Tikhonov functionals]
\label{thm:gamma-convergence}
As $s\to1^-$, the functionals
$\mathcal J_{\alpha,s}:L^2(\Omega)\to(-\infty,+\infty]$
$\Gamma$-converge to
$\mathcal J_{\alpha,1}:L^2(\Omega)\to(-\infty,+\infty]$
with respect to the strong topology of $L^2(\Omega)$.
\end{theorem}

\begin{proof}
The $\Gamma$-liminf inequality follows from
Proposition~\ref{prop:gamma-liminf}, while the $\Gamma$-limsup
inequality follows from Proposition~\ref{prop:gamma-limsup}. Hence
$\mathcal J_{\alpha,s} \Gamma$-converges to
\(\mathcal J_{\alpha,1}\) in \(L^2(\Omega)\).
\end{proof}

\(\Gamma\)-convergence alone does not pass to minimizers without a
matching compactness property, which we record next.

\begin{theorem}[Equicoercivity]\label{thm:equicoercive}
The family \(\{\mathcal J_{\alpha,s}\}\) is equicoercive along
\(s\to1^{-}\) in \(L^2(\Omega)\). More precisely, if \(s_n\to1^{-}\)
and \(\mathcal J_{\alpha,s_n}(u_n)\le C\), then \(\{u_n\}\) is
relatively compact in \(L^2(\Omega)\).
\end{theorem}
\begin{proof}
Let $
\mathcal J_{\alpha,s_n}(u_n)\le c, \, s_n\to1^-.$
By Lemma~\ref{lem:coercive},
$\sup_n\|u_n\|_{L^2(\Omega)}<\infty.$
Moreover,
$\sup_n \mathcal{E}_{s_n}(u_n,u_n)<\infty$.
Therefore Lemma~\ref{lem:BBMcompact} implies that $\{u_n\}$ is relatively compact in
$L^2(\Omega)$. Hence every sequence $\{u_n\}$ satisfying $s_n\to1^-$ and
$\sup_n \mathcal{J}_{\alpha,s_n}(u_n)<\infty$ is relatively compact in $L^2(\Omega)$. This proves equicoercivity along 
$s\to1^-$.
\end{proof}

$\Gamma$-convergence and equicoercivity together now yield the main
result of this subsection.

\begin{theorem}[Convergence of minimizers]
Let $u_{\alpha,s}^{\delta}$ be the unique minimizer of
$\mathcal J_{\alpha,s}$, and let $u_{\alpha,1}^{\delta}$ be the
unique minimizer of $\mathcal J_{\alpha,1}$. Then, as $s\to1^{-}$,
\[
u_{\alpha,s}^{\delta}\to u_{\alpha,1}^{\delta} \text{ strongly in }L^2(\Omega).
\]
\end{theorem}

\begin{proof}
By Theorem \ref{thm:equicoercive}, the family $\{\mathcal J_{\alpha,s}\}$ is equicoercive
in $L^2(\Omega)$. By Theorem~\ref{thm:gamma-convergence}, $\mathcal J_{\alpha,s}$
$\Gamma$-converges to $\mathcal J_{\alpha,1}$ with respect to the
strong topology of $L^2(\Omega)$. The fundamental theorem of
$\Gamma$-convergence therefore implies that every sequence of
minimizers $u_{\alpha,s}^{\delta}$ converges, up to subsequences, to a
minimizer of $\mathcal J_{\alpha,1}$. Since
$\mathcal J_{\alpha,1}$ is strictly convex, its minimizer is unique.
Hence the whole family converges to $u_{\alpha,1}^{\delta}$ strongly
in $L^2(\Omega)$.
\end{proof}

This completes the passage to the classical limit: as $s\to1^-$,
the fractional Tikhonov functionals, their minimizers, and the
underlying energies all converge to their classical $H_0^1$
counterparts. We turn next to numerical experiments illustrating the
fractional-order theory of Sections~\ref{sec:var-form}--\ref{sec:application}.

\section{Numerical experiments}
\label{sec:numerical-experiments}

We complement the analytical results with numerical experiments for the
backward fractional heat problem of Section~\ref{sec:application}: a
representative reconstruction, the effect of the penalty order, a
comparison against the classical $H_0^1$ penalty, robustness to the
noise level, and a Monte Carlo test of the mean-square rate of
Theorem~\ref{thm-heat}. Throughout, the restricted fractional Laplacian
is discretized directly from its integral form, not via the spectral
fractional power of the classical Dirichlet Laplacian, so the
experiments approximate the same nonlocal operator analyzed
throughout the paper.

\subsection{Experimental setup}
\label{sec:experimental-setup}

All computations are on $\Omega=(0,1)$ with $n=120$ interior grid
points ($h=1/(n{+}1)$). The bilinear form $a_s$ is discretized by a
rectangle-rule quadrature of the singular kernel on the interior grid,
together with the exact closed-form exterior (tail) contribution
$C_{1,s}(x^{-2s}+(1-x)^{-2s})/(2s)$; the discrete mass matrix is
$M=hI$. The resulting matrix $A_{s,h}=M+S_h$ is diagonalized once via
the generalized eigenproblem $A_{s,h}v=\lambda M v$, and
$e^{-tA_{s,h}}$ is assembled from this eigendecomposition by spectral
calculus. Synthetic data are generated by applying the discrete
semigroup to a prescribed $u^\dagger$ and adding Gaussian-direction
noise rescaled so that $\|\eta^\delta\|_{M_h}=\delta$ exactly, matching
the deterministic noise model of Theorem~\ref{thm:discrepancy}. Except
in the Monte Carlo study, $\alpha$ is selected from a fixed
logarithmic grid of $150$ candidates over $[10^{-10},10^{-2}]$ as the
smallest candidate whose residual reaches the target discrepancy,
\begin{equation}
\label{eq:discrepancy-numerics}
\|K_hu_{\alpha,h}^\delta-y_h^\delta\|_{M_h}\ge\tau\delta,
\qquad \tau=1.05,
\end{equation}
the discretized form of the crossing condition justified by
Theorem~\ref{thm:discrepancy}. Since each figure below uses an
independent noise realization, results at a shared $(s,\delta)$ across
different experiments agree in order of magnitude but not digit for
digit; only the Monte Carlo study in Section~\ref{sec:monte-carlo}
averages over realizations.

\subsection{Representative reconstruction}
\label{sec:representative-reconstruction}

For the smooth target
$u^\dagger(x)=\sin(\pi x)+0.35\sin(3\pi x)+0.15\sin(7\pi x)$, forward
order $s=0.55$, observation time $t=0.006$, and noise level
$\delta=2\times10^{-3}$, the discrepancy principle selects
$\alpha=3.55\times10^{-4}$, yielding
$\|u_{\alpha,h}^\delta-u_h^\dagger\|_{L^2}=4.01\times10^{-3}$ and
$\|u_{\alpha,h}^\delta-u_h^\dagger\|_{H_0^s}=5.03\times10^{-2}$.

\begin{figure}[t]
\centering
\begin{subfigure}[t]{0.48\textwidth}
    \centering
    \includegraphics[width=\textwidth]{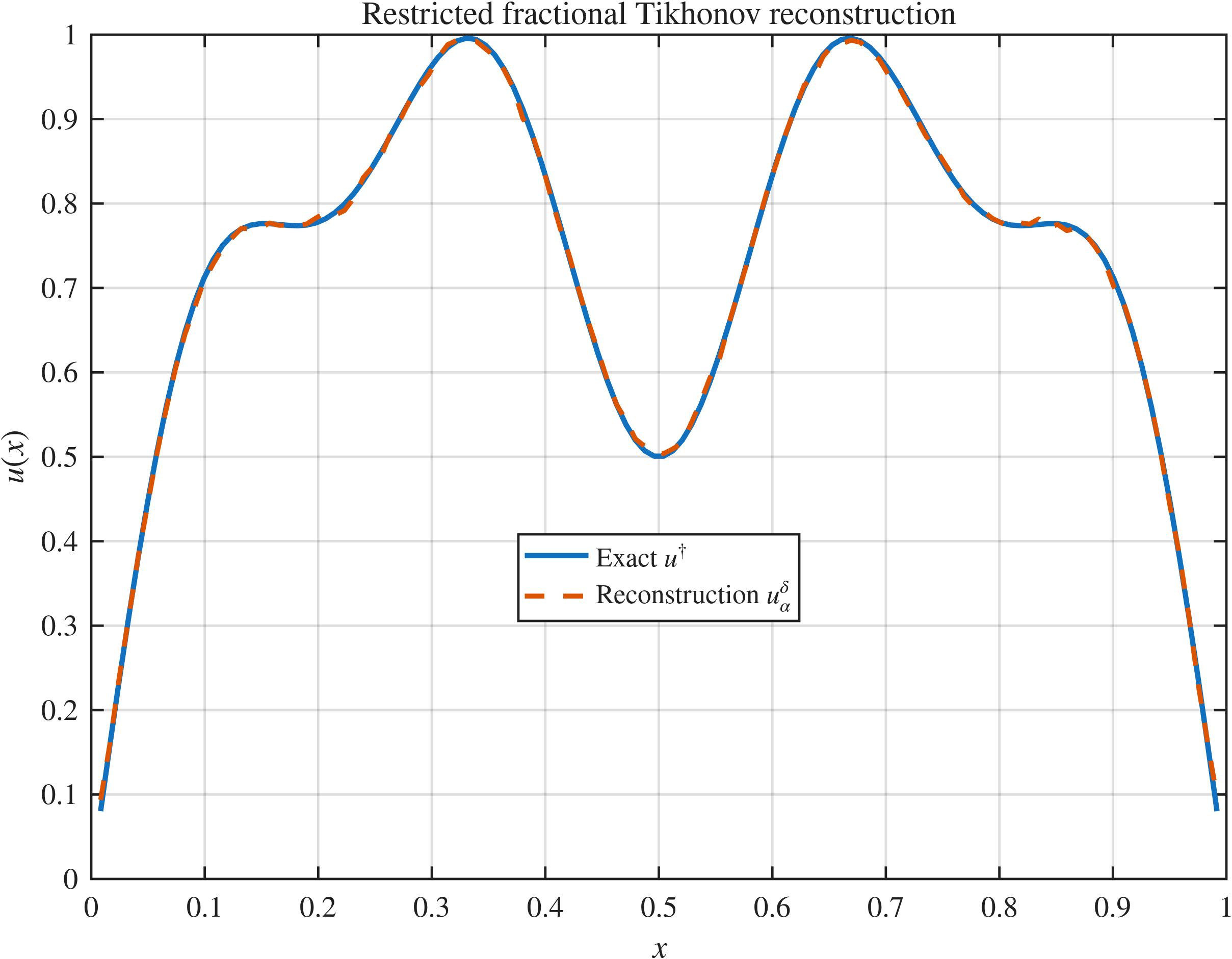}
    \caption{Exact and reconstructed solutions.}
    \label{fig:representative-reconstruction}
\end{subfigure}
\hfill
\begin{subfigure}[t]{0.48\textwidth}
    \centering
    \includegraphics[width=\textwidth]{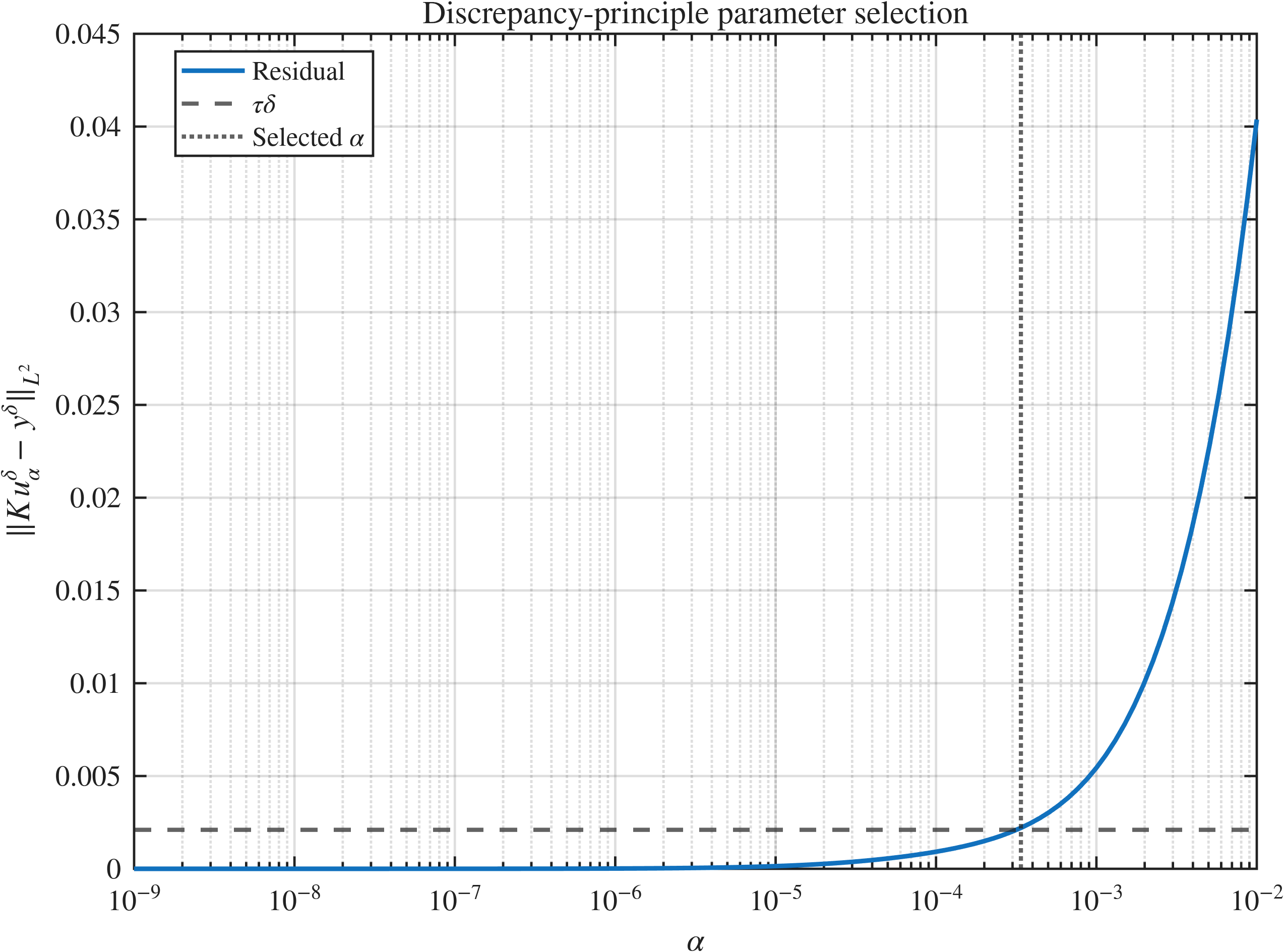}
    \caption{Residual curve and discrepancy threshold.}
    \label{fig:discrepancy-selection}
\end{subfigure}
\caption{Representative reconstruction and discrepancy-principle
parameter selection.}
\label{fig:reconstruction-discrepancy}
\end{figure}

Figure~\ref{fig:representative-reconstruction} shows close agreement
throughout $\Omega$, with the largest deviations at the local extrema,
where the oscillatory components are most sensitive to noise.
Figure~\ref{fig:discrepancy-selection} shows the residual crossing
$\tau\delta$ at the selected $\alpha$, confirming that the parameter is
fixed by the data rather than by visual tuning.

\subsection{Influence of the fractional penalty order}
\label{sec:penalty-order}

Keeping the forward operator, data, and discrepancy rule fixed, we vary
only the penalty order $s_{\mathrm{pen}}\in\{0.25,0.40,0.55,0.70,0.85\}$
against the same $s=0.55$ forward problem. Only $s_{\mathrm{pen}}=0.55$
is matched in the sense of Theorem~\ref{thm-heat}; the remaining cases
illustrate the modeling flexibility of a tunable penalty order rather
than the matched-rate theory.

\begin{figure}[t]
\centering
\begin{subfigure}[t]{0.48\textwidth}
    \centering
    \includegraphics[width=\textwidth]{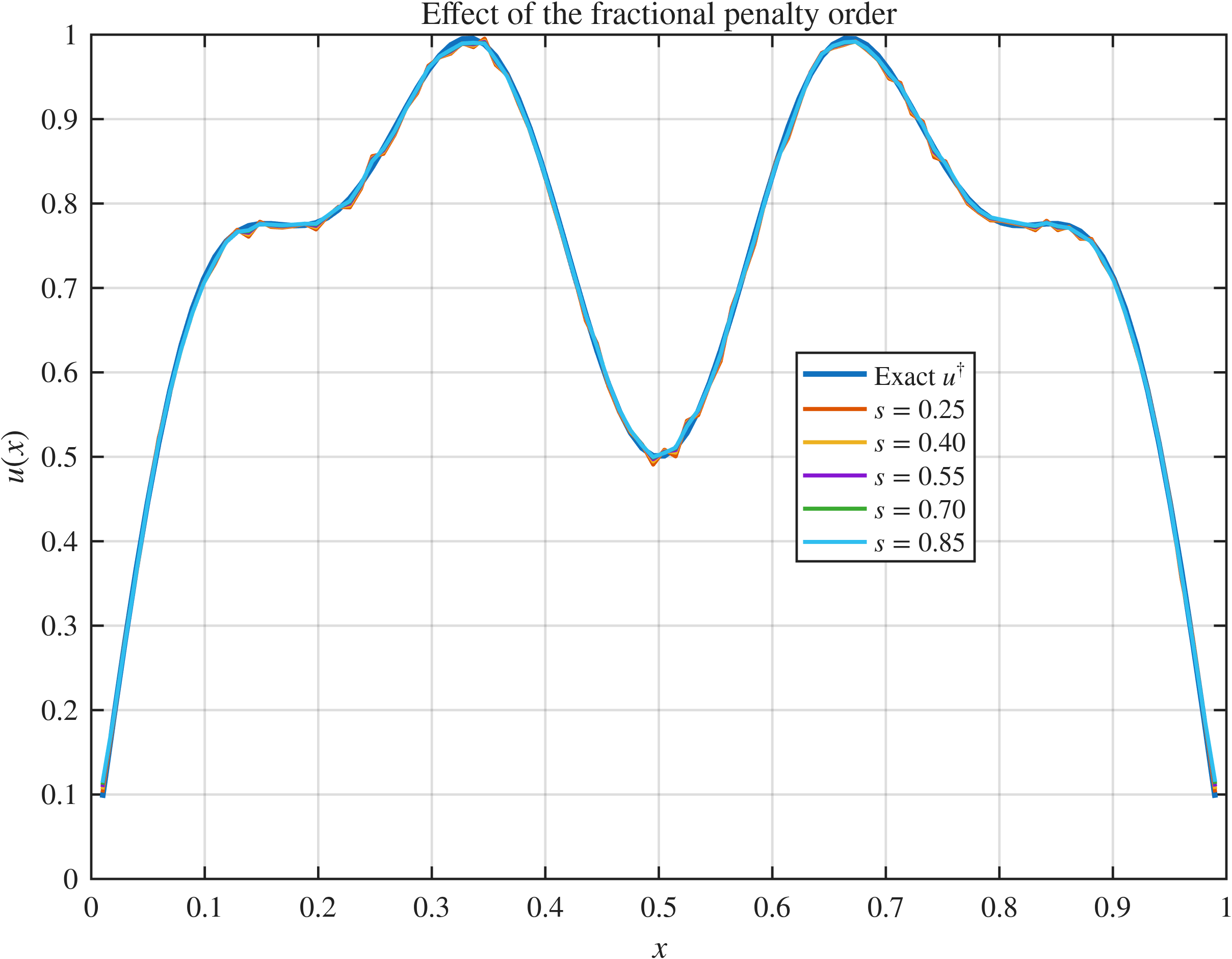}
    \caption{Reconstructions for each penalty order.}
    \label{fig:fractional-order-profiles}
\end{subfigure}
\hfill
\begin{subfigure}[t]{0.48\textwidth}
    \centering
    \includegraphics[width=\textwidth]{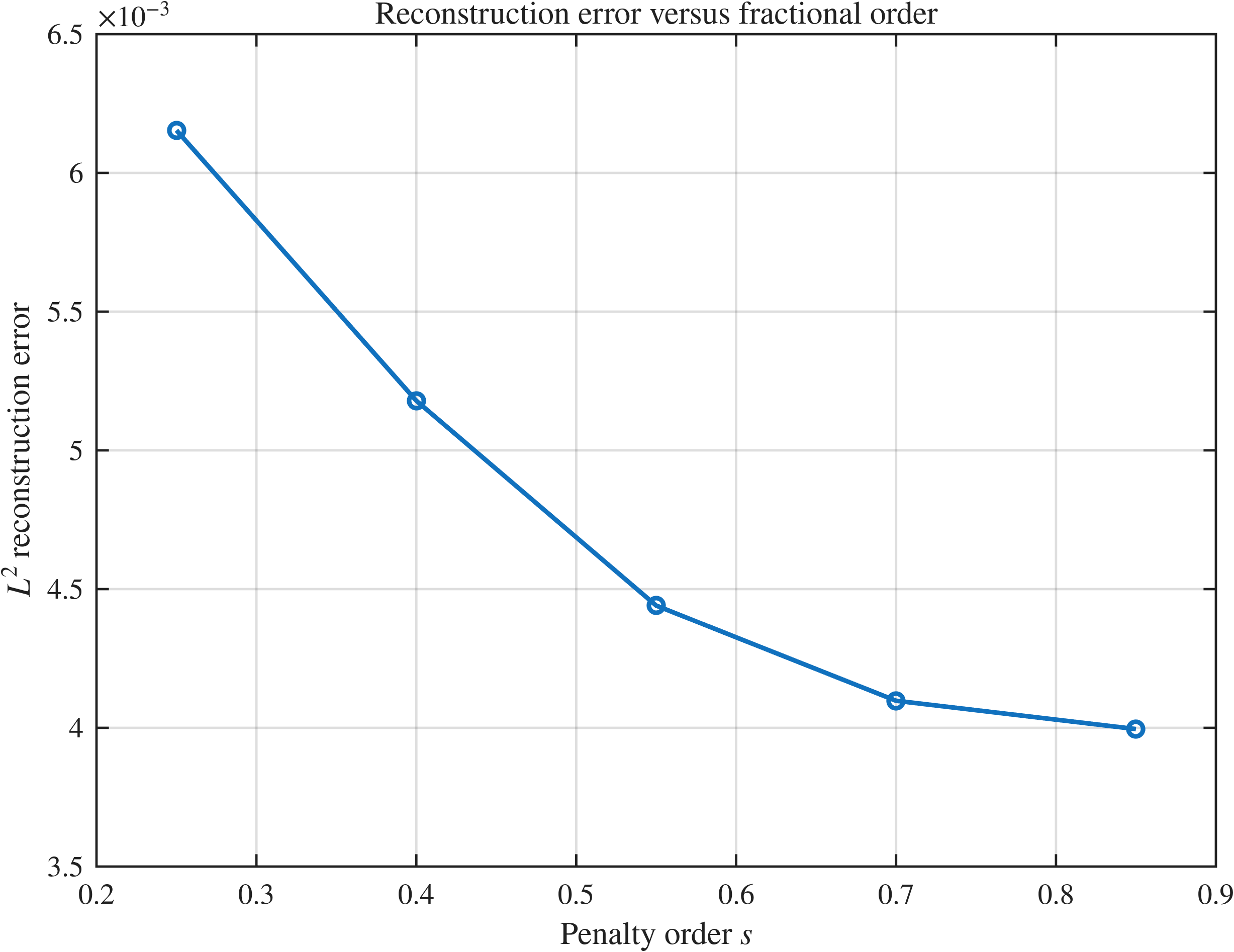}
    \caption{$L^2$ error versus penalty order.}
    \label{fig:fractional-order-errors}
\end{subfigure}
\caption{Influence of the fractional penalty order on reconstruction
quality.}
\label{fig:fractional-order-study}
\end{figure}

\begin{table}[t]
\centering
\caption{Influence of the fractional penalty order.}
\label{tab:fractional-order}
\begin{tabular}{ccc}
\toprule
$s_{\mathrm{pen}}$ & Selected $\alpha$ & $L^2$-error\\
\midrule
0.25 & $1.08\times10^{-3}$ & $6.70\times10^{-3}$\\
0.40 & $6.59\times10^{-4}$ & $5.07\times10^{-3}$\\
0.55 & $3.55\times10^{-4}$ & $4.24\times10^{-3}$\\
0.70 & $1.49\times10^{-4}$ & $3.90\times10^{-3}$\\
0.85 & $8.05\times10^{-5}$ & $3.87\times10^{-3}$\\
\bottomrule
\end{tabular}
\end{table}

All five reconstructions are visually stable
(Figure~\ref{fig:fractional-order-profiles}), while
Table~\ref{tab:fractional-order} shows the $L^2$-error decreasing
monotonically with $s_{\mathrm{pen}}$, from $6.70\times10^{-3}$ to
$3.87\times10^{-3}$, and nearly plateauing beyond
$s_{\mathrm{pen}}\approx0.7$. The penalty order is thus a genuine
modeling knob: it adjusts the regularization continuously rather than
forcing a choice between the classical first-order penalty and one
fixed fractional order.

\subsection{Comparison with the classical $H_0^1$ penalty}
\label{sec:classical-comparison}

Both methods share the same forward operator, grid, noisy data
($\delta=3\times10^{-3}$), and discrepancy rule; only the penalty
differs. The target is the localized multiscale profile
$u^\dagger(x)=e^{-90(x-0.34)^2}+0.70\,e^{-170(x-0.73)^2}+0.10\sin(9\pi x)$.

\begin{figure}[t]
\centering
\includegraphics[width=0.5\textwidth]{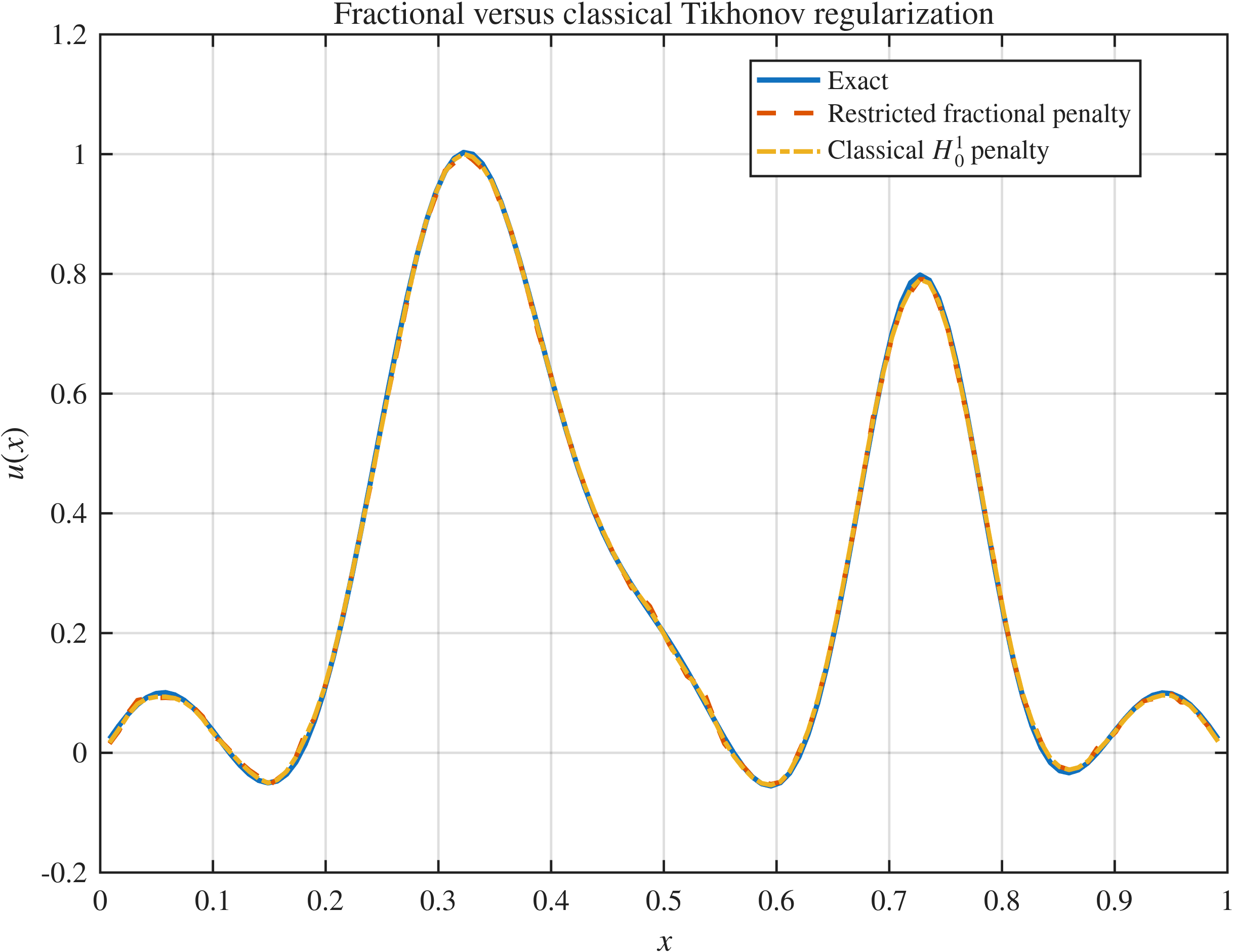}
\caption{Restricted fractional versus classical $H_0^1$ Tikhonov
regularization.}
\label{fig:fractional-classical}
\end{figure}

\begin{table}[t]
\centering
\caption{Restricted fractional versus classical $H_0^1$ regularization.}
\label{tab:classical-comparison}
\begin{tabular}{ccc}
\toprule
Method & Selected $\alpha$ & $L^2$-error\\
\midrule
Restricted fractional ($s=0.55$) & $3.55\times10^{-4}$ & $5.80\times10^{-3}$\\
Classical $H_0^1$ & $1.83\times10^{-5}$ & $4.26\times10^{-3}$\\
\bottomrule
\end{tabular}
\end{table}

Both penalties recover the principal features accurately
(Figure~\ref{fig:fractional-classical}); the classical penalty attains
a somewhat smaller $L^2$-error on this particular smooth-plus-localized
target (Table~\ref{tab:classical-comparison}). We read this not as a
claim of superiority but as evidence that the fractional penalty is
competitive with the classical one while offering an additional,
continuously tunable order $s$ -- an advantage precisely when the
target's regularity is nonlocal or not known a priori.

\subsection{Robustness with respect to the noise level}
\label{sec:noise-robustness}

With all other parameters as in Section~\ref{sec:representative-reconstruction},
$\alpha$ is reselected independently at each
$\delta\in\{10^{-2},5\times10^{-3},2\times10^{-3},10^{-3}\}$.

\begin{figure}[t]
\centering
\begin{subfigure}[t]{0.48\textwidth}
    \centering
    \includegraphics[width=\textwidth]{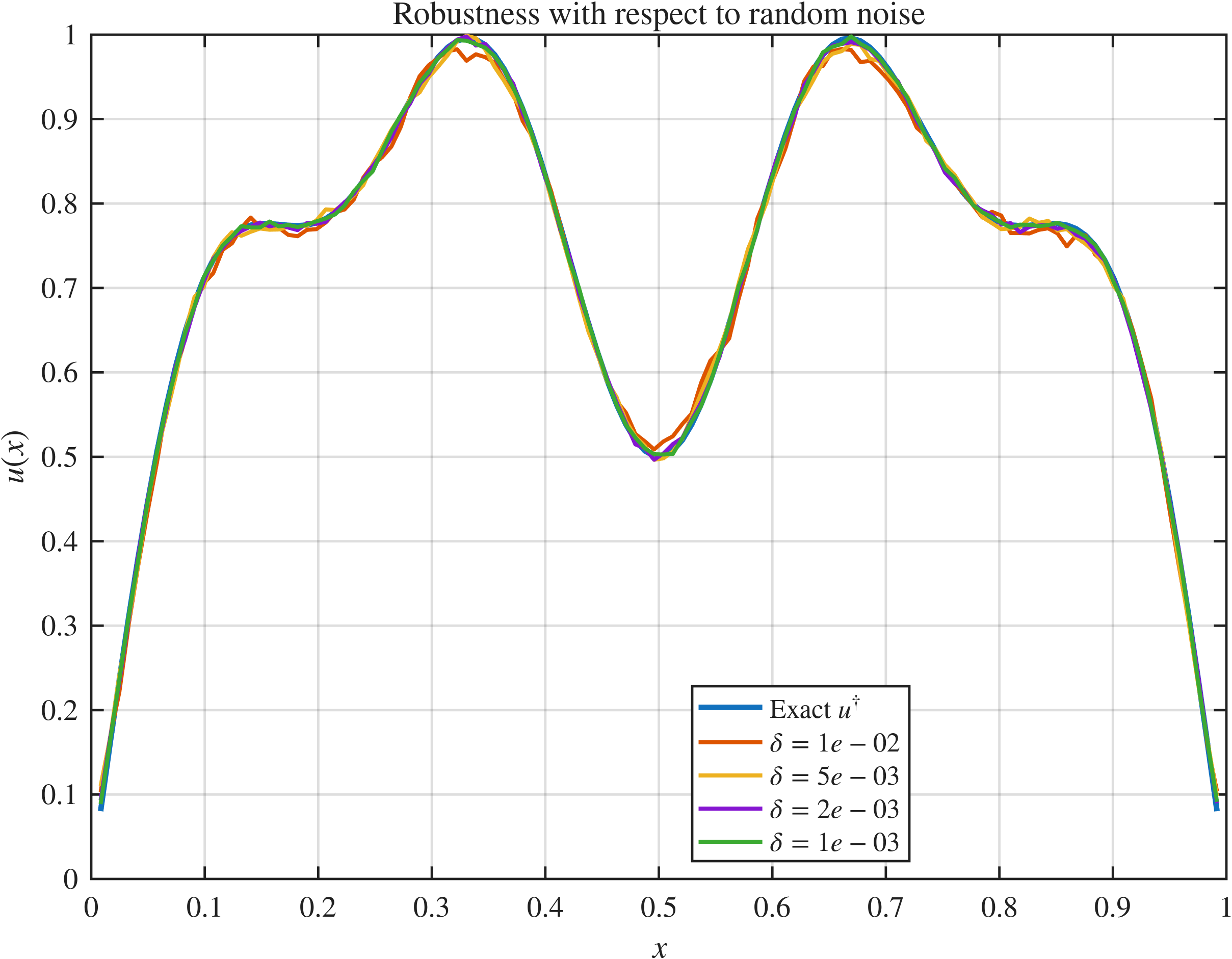}
    \caption{Reconstructions at each noise level.}
    \label{fig:noise-profiles}
\end{subfigure}
\hfill
\begin{subfigure}[t]{0.48\textwidth}
    \centering
    \includegraphics[width=\textwidth]{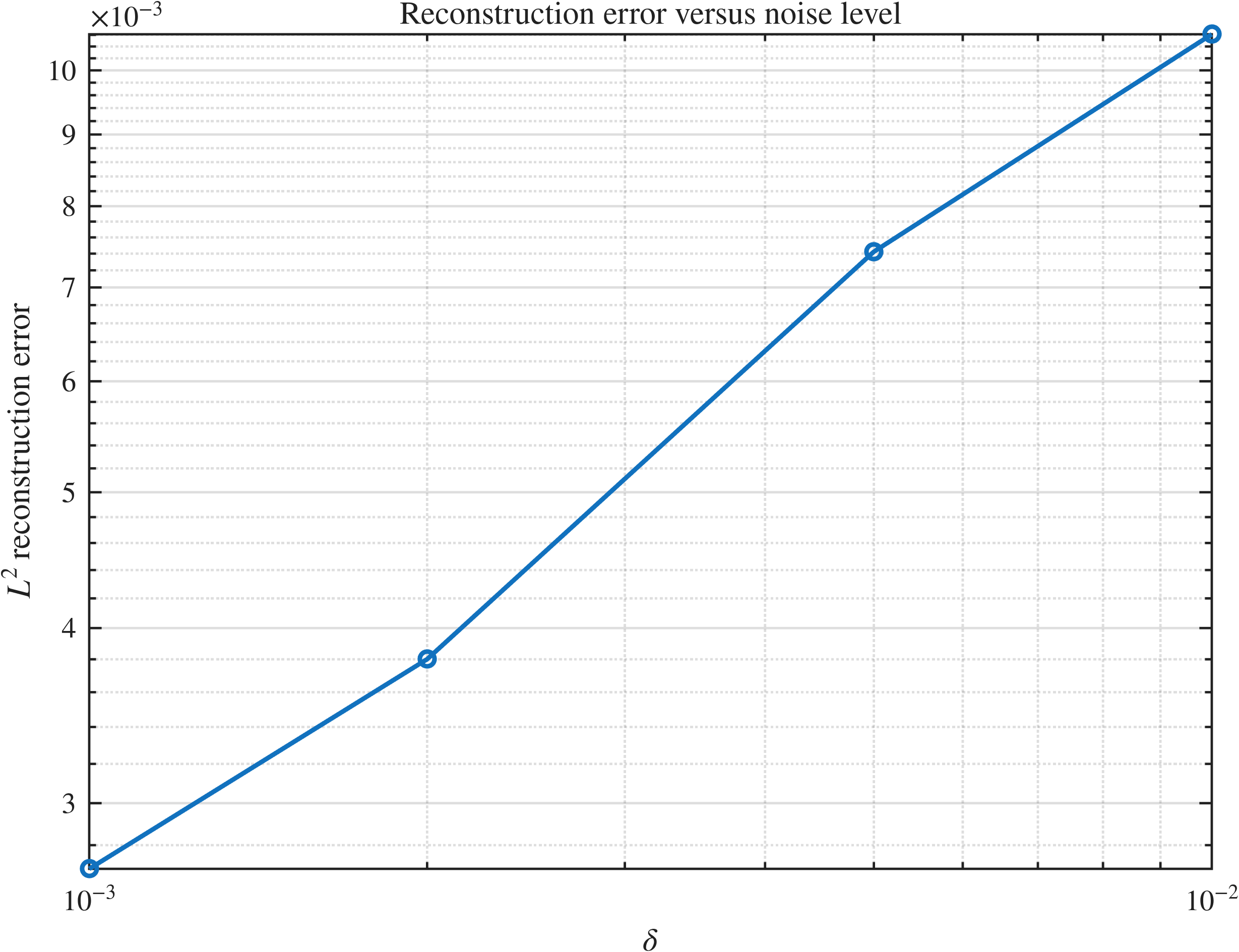}
    \caption{$L^2$-error versus $\delta$.}
    \label{fig:noise-errors}
\end{subfigure}
\caption{Robustness of the reconstruction to the noise level.}
\label{fig:noise-study}
\end{figure}

\begin{table}[t]
\centering
\caption{Reconstruction error versus noise level.}
\label{tab:noise-robustness}
\begin{tabular}{ccc}
\toprule
$\delta$ & Selected $\alpha$ & $L^2$-error\\
\midrule
$1.0\times10^{-2}$ & $1.57\times10^{-3}$ & $1.11\times10^{-2}$\\
$5.0\times10^{-3}$ & $8.44\times10^{-4}$ & $7.60\times10^{-3}$\\
$2.0\times10^{-3}$ & $3.14\times10^{-4}$ & $3.79\times10^{-3}$\\
$1.0\times10^{-3}$ & $1.91\times10^{-4}$ & $2.72\times10^{-3}$\\
\bottomrule
\end{tabular}
\end{table}

The reconstruction degrades gracefully as $\delta$ grows
(Figure~\ref{fig:noise-study}): the error decreases monotonically with
$\delta$ (Table~\ref{tab:noise-robustness}), and the discrepancy
principle automatically selects a smaller $\alpha$ as the data become
more accurate, exactly as the theory of
Section~\ref{sec:statistical-error} predicts.

\subsection{Monte Carlo verification of the convergence rate}
\label{sec:monte-carlo}

Finally, we test the mean-square rate of Theorem~\ref{thm-heat}
directly. We fix $s=0.55$, $t=0.004$, and a source exponent
$\beta=0.75$, for which
\[
\mathbb E\|u_\alpha^\delta-u^\dagger\|_{H_0^s(\Omega)}^2
=O\bigl(\delta^{4\beta/(2\beta+1)}\bigr)=O(\delta^{1.2}).
\]
The exact discrete solution $u^\dagger$ is built so that the discrete
source condition $z^\dagger=(B^*B)^\beta w$ holds exactly in the
generalized eigenbasis, and $\alpha=c_\alpha\delta^{2/(2\beta+1)}$ is
used with $c_\alpha$ calibrated once on an independent pilot noise
level. For each of $8$ noise levels spanning
$\delta\in[10^{-4},10^{-1}]$, the mean-square error is averaged over
$500$ independent realizations.

\begin{figure}[t]
\centering
\includegraphics[width=0.55\textwidth]{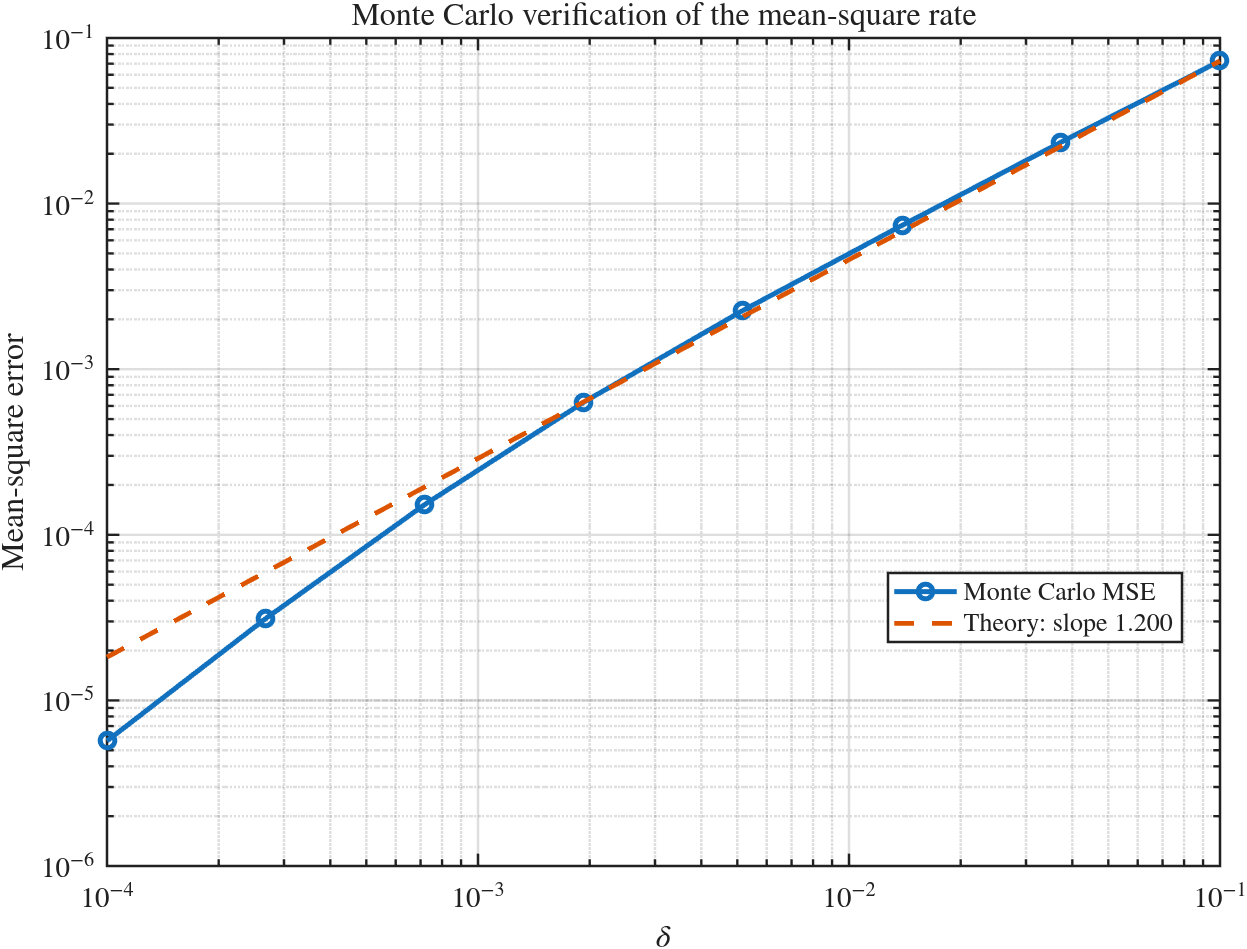}
\caption{Monte Carlo mean-square error versus the theoretical slope
$4\beta/(2\beta{+}1)=1.2$.}
\label{fig:mc-rate}
\end{figure}

\begin{table}[t]
\centering
\small
\setlength{\tabcolsep}{4pt}
\renewcommand{\arraystretch}{0.95}
\caption{Monte Carlo results (500 realizations per noise level). EOC is
the experimental order of convergence between consecutive rows.}
\label{tab:monte-carlo}
\begin{tabular}{ccccc}
\toprule
$\delta$ & $\alpha$ & MSE & Std.\ error & EOC\\
\midrule
$1.0\times10^{-4}$ & $1.12\times10^{-4}$ & $1.08\times10^{-5}$ & $3.71\times10^{-8}$ & --\\
$2.68\times10^{-4}$ & $2.47\times10^{-4}$ & $5.08\times10^{-5}$ & $1.48\times10^{-7}$ & 1.565\\
$7.20\times10^{-4}$ & $5.44\times10^{-4}$ & $1.98\times10^{-4}$ & $4.44\times10^{-7}$ & 1.377\\
$1.93\times10^{-3}$ & $1.20\times10^{-3}$ & $6.55\times10^{-4}$ & $1.36\times10^{-6}$ & 1.215\\
$5.18\times10^{-3}$ & $2.64\times10^{-3}$ & $2.00\times10^{-3}$ & $4.92\times10^{-6}$ & 1.131\\
$1.39\times10^{-2}$ & $5.81\times10^{-3}$ & $6.03\times10^{-3}$ & $2.14\times10^{-5}$ & 1.117\\
$3.73\times10^{-2}$ & $1.28\times10^{-2}$ & $1.80\times10^{-2}$ & $8.79\times10^{-5}$ & 1.111\\
$1.0\times10^{-1}$ & $2.82\times10^{-2}$ & $5.37\times10^{-2}$ & $3.39\times10^{-4}$ & 1.106\\
\bottomrule
\end{tabular}
\end{table}

Figure~\ref{fig:mc-rate} and Table~\ref{tab:monte-carlo} show the MSE
tracking the theoretical slope closely across the full two decades of
$\delta$, with the EOC decreasing monotonically from $1.565$ toward
$1.106$ as $\delta$ grows, bracketing the theoretical value $1.2$ from
both sides. A least-squares fit over the six interior noise levels
(excluding both extremes, where discretization and coarse-noise
effects respectively are most likely to intrude) gives a fitted
exponent $p_{\mathrm{fit}}=1.180$, within $2\%$ of the theoretical rate
$p_{\mathrm{th}}=4\beta/(2\beta+1)=1.200$ -- strong numerical
confirmation of Theorem~\ref{thm-heat} and
Corollary~\ref{cor:optimal_parameter}.

\bigskip
Together, these experiments show that the restricted fractional
Tikhonov method reconstructs accurately and stably, that Morozov's
discrepancy principle is a reliable and fully automatic
parameter-selection rule as Theorem~\ref{thm:discrepancy} predicts, and
that the Monte Carlo results confirm the theoretical mean-square rate
to within a few percent.

\section{Conclusion}\label{sec:conclusion}

We developed a fractional-Sobolev Tikhonov framework for linear inverse
problems based on the restricted Dirichlet fractional energy. The
associated functional is strictly convex, coercive, and weakly lower
semicontinuous, so its minimizer exists uniquely and depends Lipschitz
continuously on the data. The identification
\(D(A_s^{1/2})=H_0^s(\Omega)\), \(A_s=I+(-\Delta)^s\), provides an
isometric passage to a classical Hilbert-space Tikhonov problem,
yielding explicit mean-square error estimates and order-optimal
\emph{a priori} and \emph{a posteriori} parameter rules -- the latter
via Morozov's discrepancy principle, without requiring the
source-condition exponent in advance.

The partial-observation example shows the abstract theory applies
beyond semigroup-generated forward maps. For backward fractional heat
inversion, the identity \(B^*B=A_s^{-1}e^{-2tA_s}\) enables a detailed
spectral analysis: the source condition can be written mode by mode,
the singular values decay exponentially, and the effective number of
recoverable modes grows only polylogarithmically in the inverse noise
level, reflecting an exponential regularity condition on the exact
solution.

After Bourgain--Brezis--Mironescu normalization, the fractional
functionals \(\Gamma\)-converge in strong \(L^2(\Omega)\) to the
classical \(H_0^1\)-Tikhonov functional, with equicoercivity giving
convergence of the corresponding minimizers.

Numerical experiments corroborate the theory quantitatively: a Monte
Carlo study over \(500\) realizations per noise level gives a fitted
mean-square convergence exponent within \(2\%\) of the theoretical
rate \(4\beta/(2\beta+1)\) of Theorem~\ref{thm-heat}, and the
discrepancy principle used throughout the remaining experiments --
rather than an ad hoc heuristic -- is exactly the a posteriori rule of
Theorem~\ref{thm:discrepancy}.

Natural extensions include nonlinear and constrained inverse problems,
a posteriori theory tailored to the matched spectral structure of
Section~\ref{sec:application} (e.g.\ via the Lepskii balancing
principle), variable-order and anisotropic nonlocal penalties, and
quantitative discretization-error analysis for the restricted
fractional operator.

\medskip

\subsection*{Future directions}
The present work establishes a rigorous Hilbert-scale framework for
fractional Tikhonov regularization based on the restricted fractional
Laplacian. Several natural extensions remain open, including nonlinear
inverse problems, more general nonlocal operators, variable-order
fractional models, and adaptive parameter-choice strategies. From the
computational perspective, efficient finite-element discretizations and
fast solvers for large-scale multidimensional problems deserve further
investigation. It would also be of considerable interest to extend the
analysis to statistical and Bayesian inverse problems, thereby providing a
unified framework for deterministic and stochastic fractional
regularization. We believe that the ideas developed here provide a
foundation for these and related problems.

\bigskip


\subsection*{Funding}
This work was supported by the Department of Science and Technology
(DST), Government of India, through the INSPIRE Faculty Award, grant
no.\ IFA21-MA 158.

\subsection*{Data availability}
The simulation code used to generate the numerical results is available from the corresponding author upon reasonable request.

\subsection*{Conflict of interest}
The authors declare that they have no conflict of interest.

\newcommand{\etalchar}[1]{$^{#1}$}

\end{document}